\RequirePackage{fix-cm}
\documentclass{svjour3}                     
\smartqed  
\usepackage[margin=1in]{geometry}
\usepackage{graphicx}
\usepackage{amsmath} 
\usepackage{url}
\usepackage{authblk} 
\usepackage{amssymb}
\usepackage{amsfonts}
\usepackage{verbatim}
\usepackage{mathrsfs}
\usepackage{subfigure}
\usepackage{mathtools}
\usepackage{nicefrac}
\usepackage{xcolor}

\usepackage{array} 

\newcommand{\im}{\mathrm{i}}
\newcommand{\rr}{\mathbb{R}}

\newcommand{\Lc}{\mathcal{L}}

\usepackage{todonotes}

\newcolumntype{L}{>{$}l<{$}} 

\begin{document}

\title{Homoclinic dynamics in a spatial restricted four body problem
\thanks{The second author was partially supported by NSF grant 
DMS-1813501.
Both authors were partially supported by NSF grant DMS-1700154 
and by  the Alfred P. Sloan Foundation grant G-2016-7320}
}
%

\subtitle{blue skies into Smale horseshoes for vertical Lyapunov families}


\author{Maxime Murray    \and
        J.D. Mireles James 
}


\institute{M. Murray \at
              Florida Atlantic University, Department of Mathematical Sciences \\
              \email{mmurray2016@fau.edu}           
           \and
           J.D. Mireles James \at
              Florida Atlantic University, Department of Mathematical Sciences \\
              \email{jmirelesjames@fau.edu}
}

\date{Received: date / Accepted: date}

\maketitle

\begin{abstract}
The set of transverse homoclinic intersections for a 
 saddle-focus equilibrium in the 
planar equilateral restricted four body problem admit  
certain simple homoclinic oribts which form 
the skeleton of the complete homoclinic  
intersection -- or homoclinic web.  
In the present work the planar restricted four body problem 
is viewed as an invariant subsystem of the spatial problem, 
and the influence of this planar homoclinic skeleton on the spatial 
dynamics is studied from a numerical point of view.  
Starting from the vertical Lyapunov families
emanating from saddle focus equilibria,
we compute the stable/unstable manifolds of these spatial 
periodic orbits and look for intersections 
between these manifolds near the fundamental planar homoclinics.
In this way we are able to continue all of the basic planar 
homoclinic motions into the spatial problem as 
homoclinics for appropriate vertical Lyapunov orbits which, 
by the Smale Tangle theorem, suggest the existence of
chaotic motions in the spatial problem.
While the saddle-focus equilibrium solutions in the planar problems occur only at 
a discrete set of energy levels, the cycle-to-cycle homoclinics in the spatial problem are
robust with respect to small changes in energy.    
\keywords{Gravitational $4$- body problem  \and blue sky catastrophes
\and Smale tangle \and vertical Lyapunov families \and invariant manifolds
\and boundary value problems}
\PACS{45.50.Jf	 \and 45.50.Pk \and 45.10.-b \and 02.60.Lj \and 05.45.Ac}
\subclass{70K44 \and 34C45 \and 70F15}
\end{abstract}

\tableofcontents


\newpage

\section{Introduction} \label{intro}
Connecting orbits occupy a central place in the qualitative
theory of Hamiltonian systems going back to the 
groundbreaking work of Poincar\'{e} at the dawn of the 
Twentieth Century \cite{MR1194624,MR1194623,MR1194622}.
Indeed Poincar\'{e}'s argument that the circular restricted three body problem 
(CRTBP) is not integrable relies crucially on the existence of a transverse 
cycle-to-cycle homoclinic -- that is, an orbit which limits in both forward and 
backward time to a periodic solution. 
Such an orbit is necessarily in the intersection of the stable and unstable 
manifolds of the periodic solution, a fact which lends the discussion 
its distinctively geometric character.   
The interested reader is referred to the lecture notes of Chenciner 
\cite{MR3329413} for a modern discussion of the theoretical and 
historical role of invariant manifolds and connecting orbits in
Poincar\'{e}'s work on the three body problem. 
In more recent times it has been shown 
that the existence of transverse homoclinic 
orbits/heteroclinic cycles implies the existence of chaotic 
motions quite generally, via the mechanism of Smale
\cite{MR0228014}.  The Lectures of Siegel and Moser
 \cite{MR1829194,MR1345153} 
provide a classic reference on chaotic motions in celestial mechanics.

Inspired by the work of Poincar\'{e}, a number of
of early Twentieth Century dynamical astronomers --
in particular the groups led by  
Darwin, Str\"{o}mgren, and Moulton  --  
conducted extensive numerical studies 
which explored the phase space structure
of the CRTBP \cite{MR1554890,stromgrenRef,moultonBook}.  
These researchers were especially interested in one parameter
families of periodic orbits (``tubes'' parameterized by energy) and 
developed numerical continuation methods to study the
global embeddings of these tubes.  
This work first suggested the importance of 
saddle-focus libration points, as it was observed 
that some families of periodic solutions appear to
accumulate to an asymptotic cycle --
what would be called in modern language a homoclinic orbit --
for a saddle-focus libration point.   This work provided numerical 
evidence for the existence of families of periodic
orbits in the three body problem which remain bounded in 
amplitude but nevertheless have period tending to infinity, 
foreshadowing the canonical work of Chazy in 1922 on the final motions of 
three body orbits \cite{MR1509241}.

The advent of digital computing in the mid Twentieth Century 
facilitated the more detailed numerical studies of Szebehely, Nacozy, and Flandern
\cite{szebehelyOnStromgren,szebehelyTriangularPoints}. 
A key observation to emerge from this work was 
that the tubes of periodic orbits mentioned at the end of the previous paragraph
appeared to change stability infinitely many times
while approaching the homoclinic.  
This result suggested complicated dynamics near the homoclinic,
anticipating the period doubling cascades of Feigenbaum.
The interested reader is referred to the work of 
Pinotsis \cite{MR875720,MR742168,MR2726383}, 
as well as the work of Henrard and Navarro \cite{MR2104214,MR1956529}
and the references therein for more complete discussion.

These developments culminated in 1973 with Henrard's proof of a 
theorem which unified the nearly one hundred years of 
numerical experiments sketched above, finally settling a conjecture of Str\"{o}mgren
about tubes of periodic orbits. 
More precisely, Henrard showed that the existence of a transverse homoclinic for a
saddle-focus equilibrium in a two degrees of freedom Hamiltonian system implies the existence of a  
tube of periodic orbits parameterized by energy accumulating to the homoclinic \cite{MR0365628}. 
Moreover the result established that as the period of the orbits goes to infinity,
their stability does indeed change infinitely many times as earlier numerical work suggested.  
This phenomena -- the so called \textit{blue sky catastrophe} in the terminology of Abraham
\cite{MR813508} -- is studied by a number of authors including 
L.P. Shilnikov, A.L. Shilnikov, and Turaev \cite{MR3253906},  and Devaney \cite{MR0431274}. Indeed
Devaney's 1976 work established that the hypotheses of Henrard's theorem imply also the existence 
of  chaotic motions in the energy level of the saddle-focus equilibrium \cite{MR0442990}.  
We refer to the works of Lerman \cite{MR998368,MR1135905} for other theoretical results and 
discussion, and to the numerical study G\'{o}mez, L\'{i}bre, and Masdemont in 
\cite{MR993815} which illuminates saddle-focus homoclinic dynamics 
associated with the $\mathcal{L}_{4,5}$ libration points in the CRTBP.

Theorems like the ones mentioned in the previous paragraph are
Hamiltonian versions of the homoclinic bifurcations studied by Shil\'{n}ikov 
\cite{MR0259275,0025-5734-10-1-A07,MR0210987}, and 
taken together paint a vivid picture of the rich dynamics 
near a transverse homoclinic connection for a saddle-focus equilibrium
in a two degree of freedom Hamiltonian system.
A natural follow up question is \textit{what, if anything, do the results about two freedom
systems just described tell us about Hamiltonian systems with three or more degrees 
of freedom?}  The question is reasonable as 
many problems in celestial mechanics have an invariant planar
subsystem due to the conservation of angular momentum.  
 
The present work considers this question in the context of  
a spatial equilateral restricted four body problem, hereafter referred to 
as the circular restricted four body problem (CRFBP).  
The equations of motion, as well as some history and basic properties 
of the problem are reviewed in Section \ref{sec:CRFBP}.  The problem is 
an excellent candidate for the present study as the homoclinic 
dynamics in the invariant planar subsystem have recently been studied in some detail.  
In particular, the work of Shane Kepley and the second author \cite{MR3919451} 
provides a detailed numerical study of blue sky catastrophes in the case of equal masses. 
The main observation  is that the saddle-focus homoclinics 
appear to be organized by a small number of simple connections, 
or homoclinic channels.  In fact these channels are just the 
``shortest'' homoclinics (see \cite{MR3919451} for the precise meaning of 
shortest in this context), and there turn out to be six of them 
at each saddle-focus equilibrium in the CRFBP.  
If one considers these six shortest homoclinic connections as the letters of a symbolic
alphabet, then all the other homoclinic connections -- of which there 
appear to be infinitely many -- organize themselves into ``words'' in this alphabet. 
In short the \textit{homoclinic web} at any saddle focus in the CRFBP
appears to be organized by six fondamental motions.  
The results of \cite{MR3919451} are reviewed in  Section \ref{sec:Planar}.  

When the planar CRFBP is viewed as a subsystem of the spatial CRFBP the spectrum of
a libration point picks up an additional center direction, and there is an out of plane family of 
periodic orbits associated with each of the planar libration points.
These are the so called \textit{vertical 
Lyapunov families} and they inherit the stability of the planar librations. 
 We are particularly interested in the vertical families associated with the 
saddle-focus equilibrium solutions, where
the stable/unstable manifolds of the vertical periodic orbits are three dimensional
with complex conjugate stable/unstable Floquet multipliers.

The system conserves the so called \textit{Jacobi integral}, so that any fixed vertical 
Lyapunov orbit and its attached three dimensional stable/unstable manifolds live in a five 
dimensional level set, and the dimension count allows for the possibility of 
transverse intersections between the stable/unstable manifolds relative 
to the integral manifold.  If such an intersection actually occurs, it follows from the 
Smale tangle theorem \cite{MR0228014} that there is a chaotic subsystem 
near the connecting orbit.  

The present work provides compelling numerical evidence in support of the claim that
the planar homoclinic orbits studied in \cite{MR3919451} give rise to transverse homoclinic 
orbits, and hence Smale tangles, for the corresponding vertical Lyapunov families in the 
spatial CRFBP. 
That is, we find a six letter homoclinic alphabet for the spatial cycle-to-cycle connections,
inherited from the planar problem.  The out of plane connections 
appear to persist for fairly large out of plane amplitudes.

In all of our computations we utilize the \textit{parameterization method}
 to approximate the stable/unstable manifolds of the 
Lyapunov orbits in large region surrounding the periodic orbit.  
The parameterization 
method is reviewed in Section \ref{sec:parameterizationMethod}, in particular a number of 
references to the literature  are given there.
Connecting orbits are then located as solutions of two point boundary 
value problems with boundary conditions projected onto the parameterization
of the local stable/unstable manifolds.
The virtue of using the parameterization method in the present context 
is that it stabilizes the numerics, leading to a better
condition number in the two point boundary value problem.  
This is especially valuable when, as in the present work, we are trying to find
many connections as the same local parameterizations can be used to formulate
the BVPs for all the connecting orbits.
 
The remainder of the paper is organized as follows.
In Section \ref{sec:CRFBP} we review the CRFBP and 
discuss some basic results from the literature.  
In particular we review the findings of  \cite{MR3919451} on 
homoclinic channels in the planar problem and also introduce the 
vertical Lyapunov families which are the main objects of the present 
study.  In Section \ref{sec:parameterizationMethod} we review the 
parameterization method and derive the homological equations
which determine the Fourier-Taylor coefficients of the local 
invariant manifold approximations and discuss our implementations.
Section \ref{sec:connections} describes briefly the formulation of the two point 
boundary value problems for cycle-to-cycle connections, and in Section    
\ref{sec:results} we present the main results of the paper -- numerical 
calculations of the homoclinic connections for the the vertical Lyapunov
family in the CRFBP.  In Section    
\ref{sec:conclusions} we summarize our conclusions.
We provide two appendices.  One is Section \ref{sec:autoDiff}
describing the ``automatic differentiation''
framework which reduces the problem to polynomial, 
and the other is Section \ref{sec:orbitData} which tabulates
for the sake of reproducibility
some of the data produced in the present study.

\section{The restricted four body problem} \label{sec:CRFBP}
In this section, we introduce the particular version of the 
four body problem studied in the present work.  
We postpone to Section \ref{eq:CRFBP_equationsOfMotion}
discussion of the equations of motion, and 
give first a brief overview of the literature surrounding the problem, 
which originates from the work of Pedersen \cite{pedersen1,pedersen2}.
Detailed numerical studies of the equilibrium set as well as
the planar and spatial Hill's regions are found in the works of
Sim\'{o} \cite{MR510556},
 Baltagiannis and Papadakis \cite{MR2845212}, and
\'{A}lvarez-Ram\'{i}erz and Vidal \cite{MR2596303}.
Mathematically rigorous theorems about the equilibrium set and 
its bifurcations are  proven with computer assistance by Leandro and Barros in 
\cite{MR2232439,MR2784870,MR3176322}.
The papers just cited establish that for any value of the primary masses there are 
always either 8, 9, or 10 equilibrium solutions (or libration points) with 6 outside the 
equilateral triangle formed by the primary bodies
 (see Figure \ref{fig:rotatingframe}).

Periodic orbits are studied by 
Papadakis in \cite{MR3571218,MR3500916},
and by Burgos-Garc\'{i}a, Bengochea,  and Delgado in 
\cite{burgosTwoEqualMasses,MR3715396}.
A computer assisted study by Burgos-Garc\'{i}a, Lessard, and Mireles James 
proves the existence of some spatial periodic orbits for the CRFBP \cite{jpJaimeAndMe}.
Regularization of collisions are studied by \'{A}lvarez-Ram\'{i}rez, Delgado, and Vidal in 
\cite{MR3239345}.  Chaotic motions were studied numerically by Gidea and Burgos in \cite{MR2013214}, 
and by \'{A}lvarez-Ram\'{i}rez and Barrab\'{e}s in \cite{MR3304062}.

Perturbative proofs of the existence of chaotic motions are
found in the work of She, Cheng and Li
\cite{MR3626383,MR3158025,MR3038224},
and also in the work of 
Alvarez-Ram\'\i rez,  Garc\'i{a},  Palaci\'{a}n,  and Yanguas
\cite{chaosCRFBP}.
A Hill's problem is derived from the CRFBP and its periodic orbits are studied 
by Burgos-Garc\'{i}a and Gidea in \cite{MR3554377,MR3346723}.

Blue sky catastrophes in the CRFBP are studied by 
Burgos-Garc\'{i}a and Delgado in \cite{MR3105958}, and by 
Kepley and Mireles James in \cite{shaneAndJay}. This last reference 
develops computer assisted methods of proof for 
verifying the hypotheses of the theorems of Hernard and Devaney.
The authors of the last reference cited further study the blue
sky catastrophes for the CRFBP in  \cite{MR3919451}, as 
discussed already in the introduction.

\begin{figure}[!t]
\centering
\includegraphics[width=5.5in]{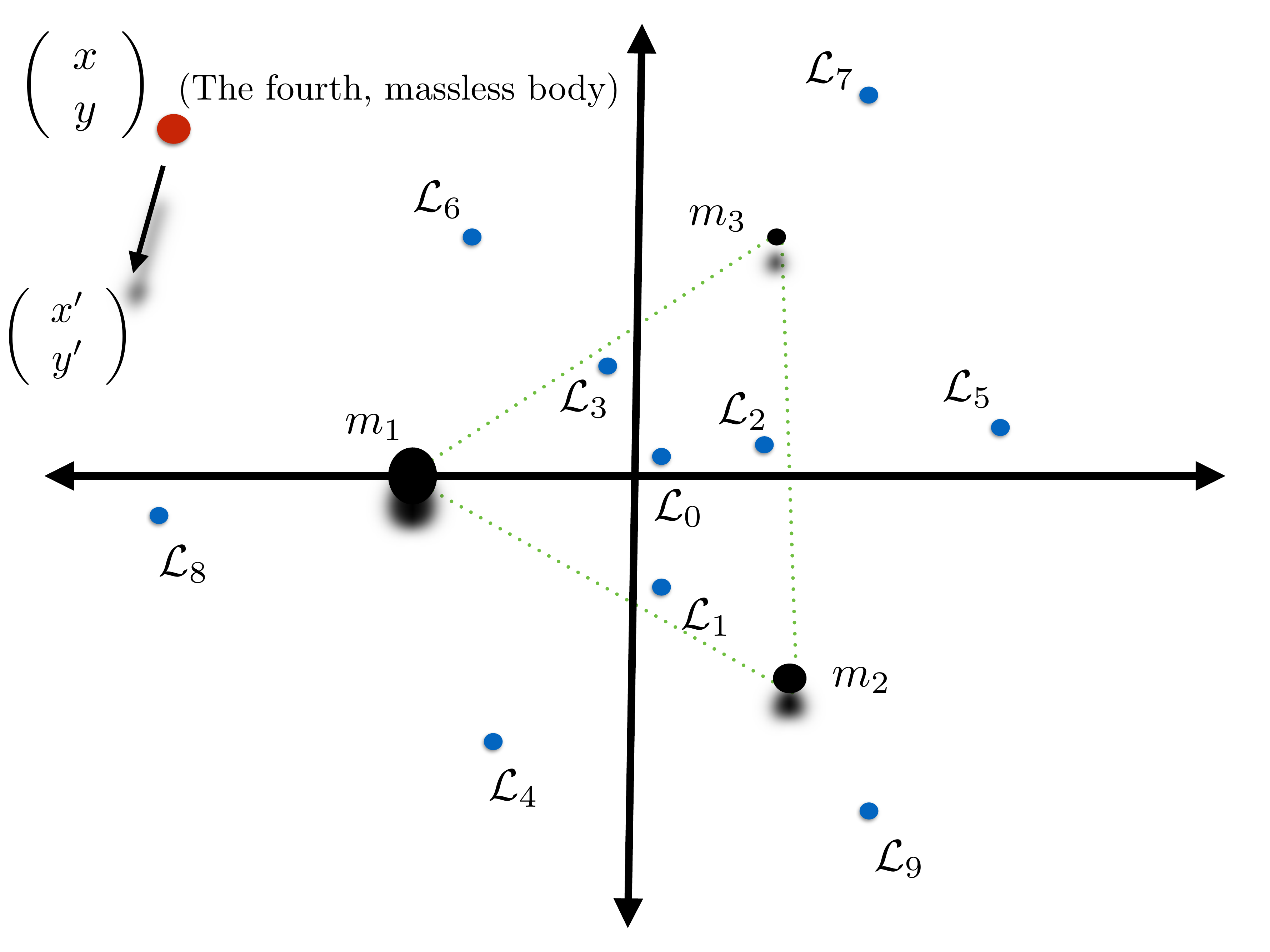}
\caption{\textbf{Configuration space for the CRFBP:} The three 
primary bodies with masses $m_1,m_2,$ and $m_3$ are 
arranged in the equilateral triangle configuration of 
Lagrange --  a relative equilibrium solution of the 
three body problem.    Transforming to a
co-rotating frame one considers the motion of a fourth
massless body.  The equations of motion have 
$8$, $9$, or $10$ equilibrium solutions
denoted by $\mathcal{L}_j$
for $0 \leq j \leq 9$.  The number of libration points,
and their stability, varies depending on $m_1$, $m_2$, and $m_3$. 
In this work we study the points $\mathcal{L}_{0, 4,5,6}$ which 
have saddle focus stability when $m_1 \approx m_2 \approx m_3$.
}\label{fig:rotatingframe}
\end{figure}

\subsection{Equations of motion and libration points for the CRFBP} \label{eq:CRFBP_equationsOfMotion}
The problem describes the motion of a massless particle moving under the influence of three massive
bodies called the primaries.  The primaries have masses $m_1$, $m_2$ and $m_3$, and are constrained
to move in the equilateral triangle configuration of Lagrange.
The masses of the primaries are normalized so that $0 < m_3 \leq m_2 \leq m_1$,
and 
\[
m_1 + m_2 + m_3 = 1,
\] 
and the problem is studied in co-rotating coordinate system.
The coordinates are chosen so that the center of mass of the primaries is at the origin, 
the largest primary is fixed on the $x$-axis, the $x$-axis cuts the side of the triangle opposite the largest primary 
and the smallest primary is in the first quadrant. 

Under these constraints the location of the primaries is a function of only the choice 
of masses. To see this let $p_i$ denote the position of the $i-$th primary 
and write 
\[
p_1 = (x_1, y_1,z_1), \quad \quad 
p_2 = (x_2, y_2,z_2), \quad \quad
\text{and} \quad \quad
p_3 = (x_3, y_3,z_3),
\]
then
\begin{eqnarray*}
x_1 &=&   \frac{-|K| \sqrt{m_2^2 + m_2 m_3 + m_3^2}}{K} \\
y_1 &=&   0 
\end{eqnarray*}
\begin{eqnarray*}
x_2 &=&  \frac{|K|\left[(m_2 - m_3) m_3 + m_1 (2 m_2 + m_3)  \right]}{
2 K \sqrt{m_2^2 + m_2 m_3 + m_3^2}
}  \\
y_2 &=&  \frac{-\sqrt{3} m_3}{2 m_2^{3/2}} \sqrt{\frac{m_2^3}{m_2^2 + m_2 m_3 + m_3^2}},
\end{eqnarray*}
\begin{eqnarray*}
x_3 &=&  \frac{|K|}{2 \sqrt{m_2^2 + m_2 m_3 + m_3^2}}  \\
y_3 &=&  \frac{\sqrt{3}}{2 \sqrt{m_2}} \sqrt{\frac{m_2^3}{m_2^2 + m_2 m_3 + m_3^2}}, 
\end{eqnarray*}
and 
\[
z_1 = z_2 = z_3 = 0, 
\]
where 
\[
K = m_2(m_3 - m_2) + m_1(m_2 + 2 m_3).
\]
Define the potential function 
\[
\Omega(x,y,z) :=
\frac{1}{2} (x^2 + y^2) + \frac{m_1}{r_1(x,y,z)} + \frac{m_2}{r_2(x,y,z)} + \frac{m_3}{r_3(x,y,z)}, 
\]
where $r_i$ represents the distance between the massless body and the $i-$th primary, so that
\[
r_1(x,y,z) := \sqrt{(x-x_1)^2 + (y-y_1)^2 +(z-z_1)^2}, 
\]
\[
r_2(x,y,z) := \sqrt{(x-x_2)^2 + (y-y_2)^2 +(z-z_2)^2},
\]
and
\[
r_3(x,y,z) := \sqrt{(x-x_3)^2 + (y-y_3)^2 +(z-z_3)^2}.
\]
The equations of motion in the co-rotating coordinates are 
\begin{equation}\begin{split}\label{ecuacionesfinales}
\ddot{x}-2\dot{y}&=\Omega_{x},\\
\ddot{y}+2\dot{x}&=\Omega_{y},\\
\ddot{z} &= \Omega_z.
\end{split}
\end{equation}

The system admits between $8$ and $10$ equilibrium solutions depending on the 
value of mass ratio, all of them lying in the $xy-$ plane. 
Closed form formulas for the locations of the equilibrium solutions
do not exist and in practice it is necessary to numerically compute
their locations once the mass ratios are fixed.  
A schematic describing the locations of the 10 equilibrium points
along with our naming conventions in the case when the 
$m_1 \approx m_2 \approx m_3$,  is given in Figure \ref{fig:rotatingframe}.

\begin{figure}[!t]
\centering
\includegraphics[width=6.0in]{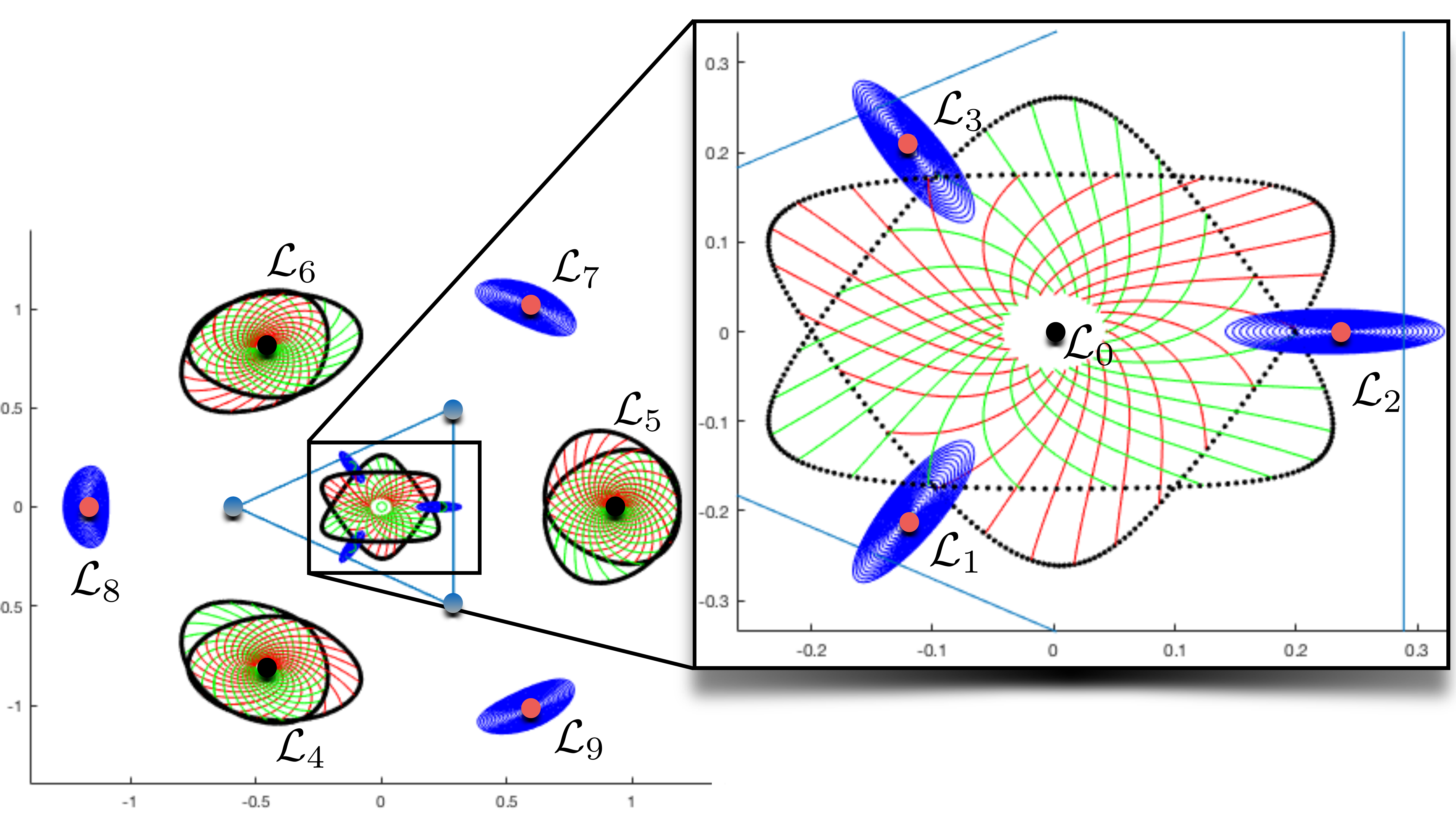}
\caption{\textbf{Local dynamics for the triple Copenhagen problem:}
when $m_1 = m_2 = m_3 = 1/3$ the libration points
$\mathcal{L}_j$ for $j = 1,2,3,7,8,9$ have center $\times$ saddle
stability, while for $j = 0, 4,5,6$ they are saddle-focus equilibria.
The figure illustrates the two dimensional center manifold in the 
former case and the two dimensional stable (green) and unstable (red)
manifolds in the latter case.  The center manifolds are populated by 
planar Lyapunov periodic orbits.  The saddle focus equilibria can exhibit
blue sky catastrophes hence are the starting point for the present work.  
}\label{fig:localDynamics}
\end{figure}

\subsection{Homoclinic dynamics in the planar problem}\label{sec:Planar}
In this section we describe the homoclinic dynamics associated with the saddle-focus
equilibrium solutions in the case that $m_1 = m_2 = m_3 = 1/3$.  We refer 
to this as \textit{the triple Copenhagen problem}, as the case of equal masses in 
the CRTBP is traditionally referred to as the Copenhagen problem.
The material in this section is discussed in much more detail in \cite{MR3919451}, on 
which the present work builds. 

The local invariant manifold structure in the triple Copenhagen problem is 
illustrated in Figure \ref{fig:localDynamics}.  The Figure depicts the fact that 
$\mathcal{L}_j$ for $j = 1,2,3,7,8, 9$ have saddle $\times$ center type stability.
Because of this, there is a planar family of Lyapunov orbits associated with 
 these libration points.  The periodic orbits foliate the attached center manifolds, 
and are illustrated by concentric blue circles in the Figure  \ref{fig:localDynamics}.  

When $j = 0, 4, 5, 6$ the libration points have saddle-focus stability.
That is, each of these libration points have a complex conjugate pair 
of stable and a complex conjugate pair of unstable eigenvalues.
The attached two dimensional stable and unstable manifolds
are foliated by orbits which converge 
exponentially to the libration point in forward and backward time.
The 2D stable/unstable orbits are illustrated by the green (unstable) 
and red (stable) curves respectively, and give a sense of the location
of the local stable/unstable manifolds.

\begin{figure}[!t]
\centering
\includegraphics[width=6.0in]{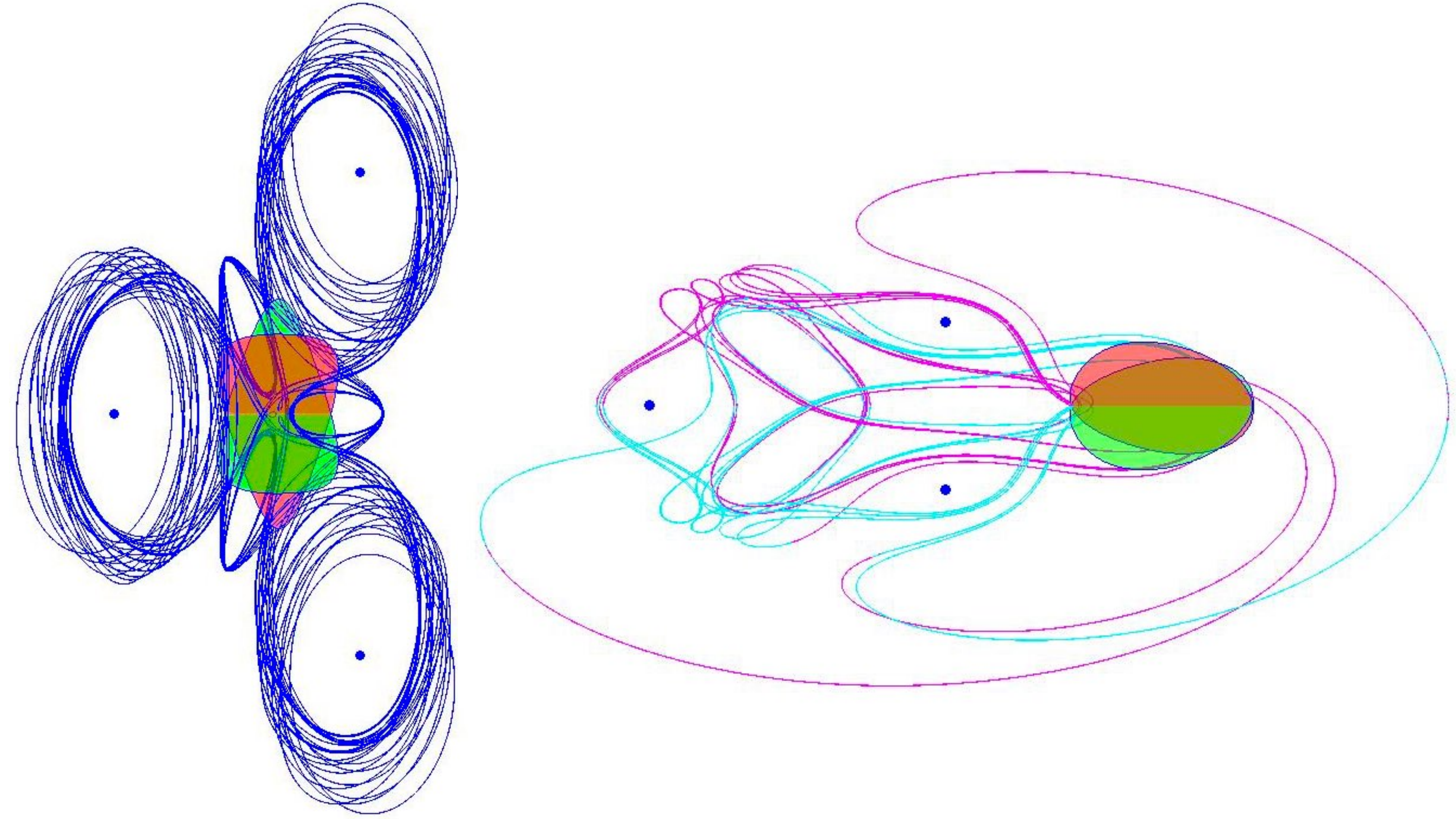}
\caption{\textbf{Homoclinic webs at $\mathcal{L}_0$ and $\mathcal{L}_5$ in the 
triple Copenhagen problem:} (Figure 9. from \cite{MR3919451}.  Reproduced with
permission of the authors) Left - 
the first $42$ homoclinic connections at $\mathcal{L}_0$.   Right - the first $23$
homoclinics at $\mathcal{L}_5$.  Local stable/unstable manifolds of the libration points
are colored in red and green respectively.  In both cases the complicated looking 
``web'' of homoclinic intersections is organized by the six simple shortest homoclinic 
motions.  See Figures \ref{fig:homoclinics_L0} and \ref{fig:homoclinics_L6}
for the six fundamental homoclinic motions at $\mathcal{L}_0$ and $\mathcal{L}_5$
respectively.  Those at $\mathcal{L}_{4,6}$ are obtained by symmetry. 
}\label{fig:homoclinicWeb}
\end{figure}

The main topic of  \cite{MR3919451} is 
to describe the geometry of the homoclinic web
-- the set of all intersections between the stable and unstable 
manifolds -- attached to the saddle focus equilibrium solutions in 
the triple Copenhagen problem.
The authors developed and deployed the following search procedure.
\begin{itemize}
\item \textbf{Step 1:} Compute high order polynomial approximations of the 
local stable/unstable manifolds. Mesh the boundary of the local approximation
into a system of one dimensional arcs. 
\item \textbf{Step 2:} Extend the local approximations by Taylor integration of the 
one dimensional boundary arcs for time $\pm\tau$.  Each step of Taylor integration results in a 
two dimensional manifold patch.  After integrating the complete system of arcs the 
result is a larger local manifold approximation.  
\item \textbf{Step 3:} Check the stable against the unstable manifold patches produced in 
step $2$ for approximate intersections.  If none are found then no intersections exist up to 
time $2 \tau$.  If an approximate intersection is found it is verified/refined using a boundary 
value solver.
\item  \textbf{Repeat:} collapse the manifold patches from Step $3$ onto their outer boundaries, 
obtaining a new system of boundary arcs for the local invariant manifold.  Then repeat 
Steps 2 and 3 as desired.  At the end of the $N$-th step the original local approximations have
been extended by time $N\tau$ thus locating all intersections up to time $2 N \tau$. 
\end{itemize}
In practice the scheme described above is combined with sophisticated step size 
selection and remeshing schemes which insure accuracy and efficiency.

As the algorithm runs all connections it locates are sorted and stored according to 
the ``time of flight'' of the orbit -- that is, the time it takes for the orbit to transition 
from the boundary of the initial unstable manifold approximation to the boundary 
of the initial stable manifold approximation.  Comparing times of flight provides a 
precise notion of ``shortest'' connections.  Complexity of the homoclinic connections
can be quantified by computing winding numbers with respect to the primary 
bodies and the libration points.

Step 1 utilizes the parameterization method as described in 
\cite{MR1976079,MR2177465}.  See also the book of 
\cite{mamotreto}.  The stable/unstable manifolds illustrated in Figure \ref{fig:localDynamics}
were computed using this method. 

Step 2 uses the methods of analytic continuation for growing atlases
of local stable/unstable invariant manifolds developed in \cite{manifoldPaper1}.
To see an illustration of how the local invariant manifolds grow see
in the triple Copenhagen problem at $\mathcal{L}_0$ and $\mathcal{L}_5$, 
see Figures 7 and 8 of  \cite{MR3919451}.
Figure \ref{fig:homoclinicWeb} illustrates the results of running the 
algorithm for $\tau = 4$ time units, hence locating all homoclinics with time 
of flight up to $8$ time units (a certain velocity constraint which removes
a small neighborhood of each of the primaries is also imposed).

The search procedure resulted in dozens of distinct homoclinic orbits at the libration points
$\mathcal{L}_0$ and $\mathcal{L}_5$.  
In the case of equal masses it is sufficient to study only these 
equilibrium solutions as a rotation by $\pm 120$ degrees transforms 
$\mathcal{L}_5$ into $\mathcal{L}_{4,6}$ respectively.
Similarly, at $\mathcal{L}_0$ rotation of any homoclinic 
by $\pm 120$ degrees yields another homoclinic
connections. 

Further examination of the connecting orbits located using the search procedure 
just described reveals the main result of  \cite{MR3919451}, which is that the homoclinic
web at each of the libration points appears to be organized by the
six shortest connections.   More precisely, each of the six shortest homoclinic 
orbits can be thought of as a letter in a symbol alphabet, and all of the homoclinic 
orbits located in the search shadow some combination of these fundamental
letters.  They are ``words'' built from the basic alphabet.  Put another way,   
only six fundamental homoclinic motions govern the complete web of connections.
The six fundamental homoclinic motions are $\mathcal{L}_0$ and $\mathcal{L}_5$
are illustrated in Figures \ref{fig:homoclinics_L0} and \ref{fig:homoclinics_L6}
respectively.

\begin{figure}[!t]
\centering
\includegraphics[width=4.75in]{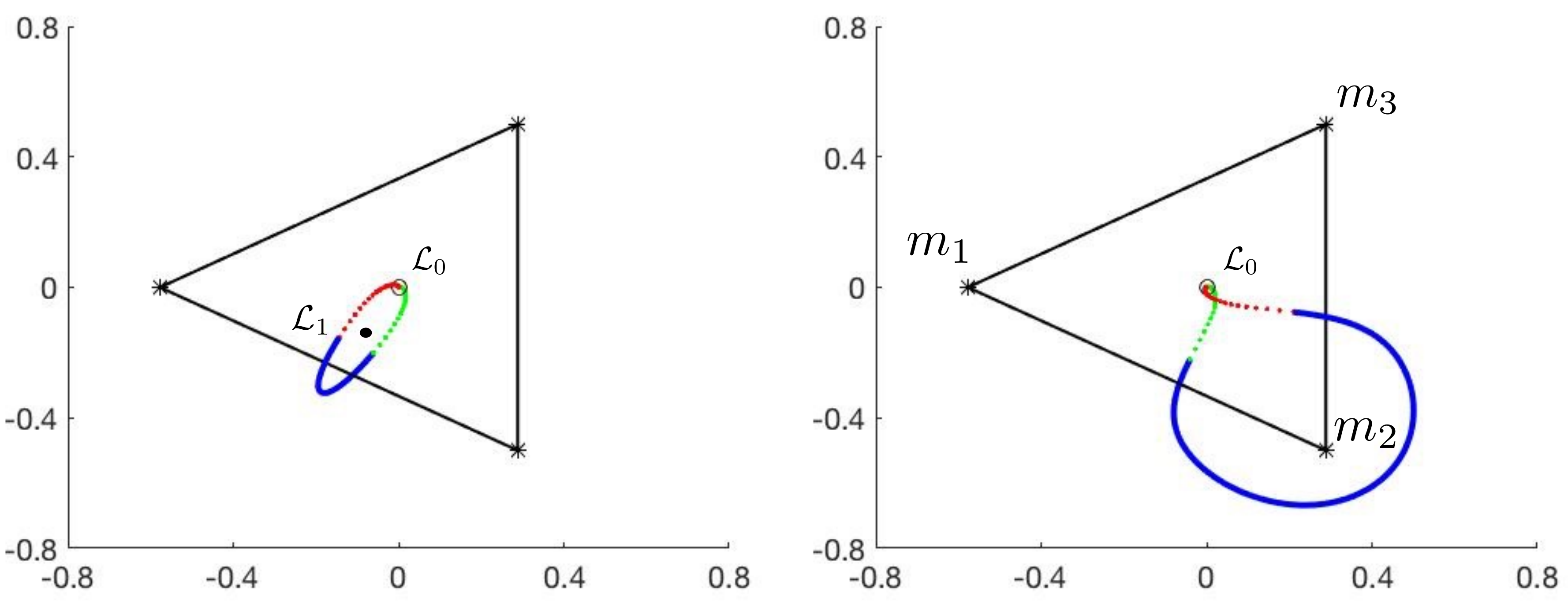}
\caption{
\textbf{Homoclinic alphabet at $\mathcal{L}_0$}: 
(Figure 10 from  \cite{MR3919451}.  Reproduced with permission of the authors).
Homoclinic motions at $\mathcal{L}_0$ with shortest times of flight.
Green and red arc segments
depict the asymptotic behavior of the homoclinic -- portion on the original stable/unstable 
manifold parameterizations.  The blue portion of the arc is the part of the orbit located by 
growing/searching the manifold atlases.  The shortest motion winds once around $\mathcal{L}_1$, 
while the second shortest motion winds once around a primary body.  Four addition basic homoclinics
are obtained by $\pm 120$ degree rotations, yielding the six letter alphabet.  The homoclinic 
web at $\mathcal{L}_0$, illustrated in the left frame of Figure \ref{fig:homoclinicWeb}, is organized by 
these six basic motions.}
\label{fig:homoclinics_L0}
\end{figure}

\begin{figure}[!h]
\centering
\includegraphics[width=4.75in]{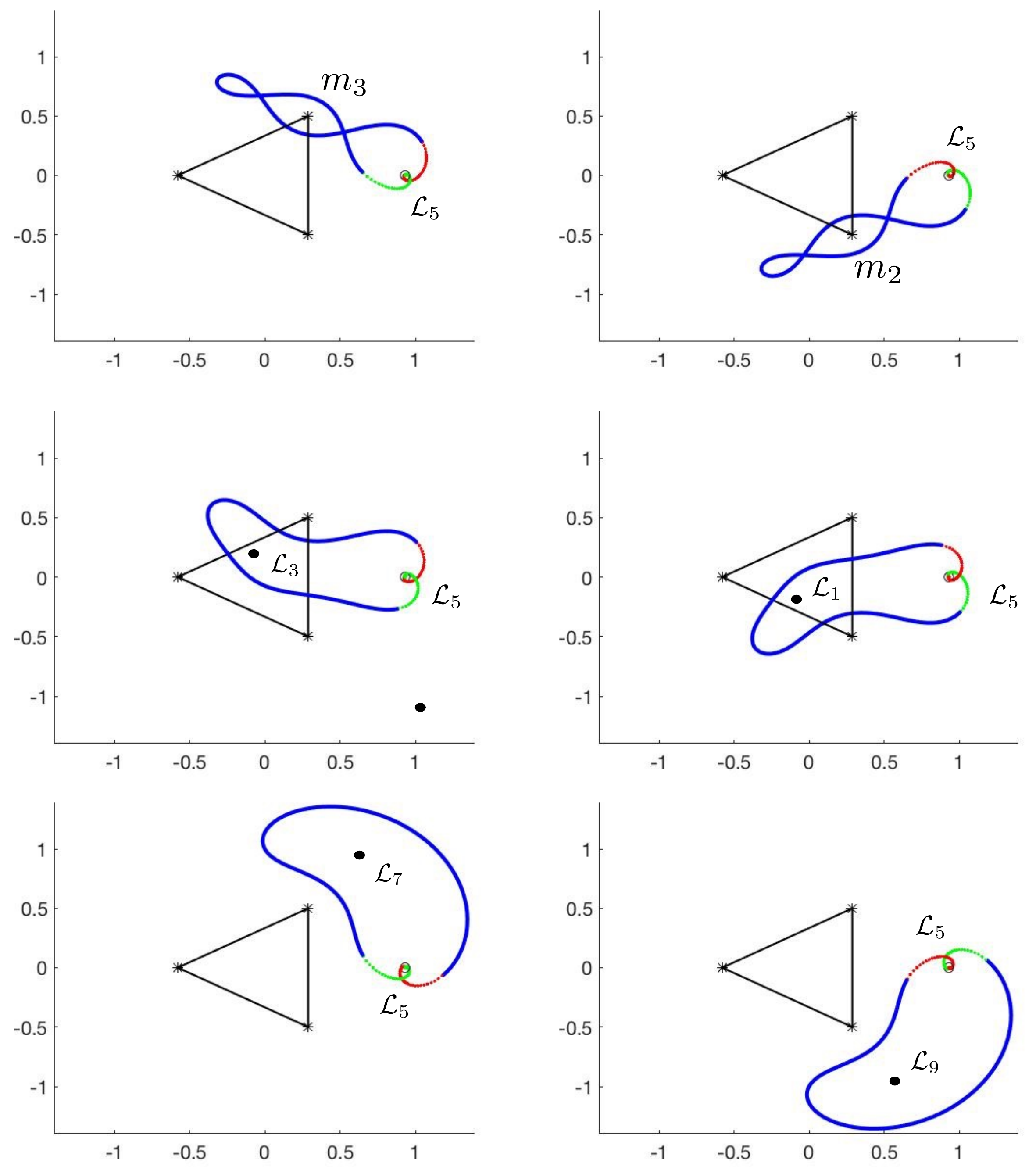}
\caption{
\textbf{Homoclinic alphabet at $\mathcal{L}_5$}: 
(Figure 18 from  \cite{MR3919451}.  Reproduced with permission of the authors).
 Homoclinic motions at $\mathcal{L}_5$ with shortest times of flight. 
  Green and red arc segments
depict the asymptotic behavior of the homoclinic -- portion on the original stable/unstable 
manifold parameterizations.  The blue portion of the arc is the part of the orbit located by 
growing/searching the manifold atlases.  Each basic motion  winds once around either a 
primary body or a libration point.  The basic motions are $\mathcal{L}_{4,6}$ are obtained by 
$\pm 120$ degree rotations. The homoclinic 
web at $\mathcal{L}_5$, illustrated in the right frame of Figure \ref{fig:homoclinicWeb}, is organized by 
these six basic motions.
}\label{fig:homoclinics_L6}
\end{figure}

\begin{figure}
\includegraphics[width=5.5in]{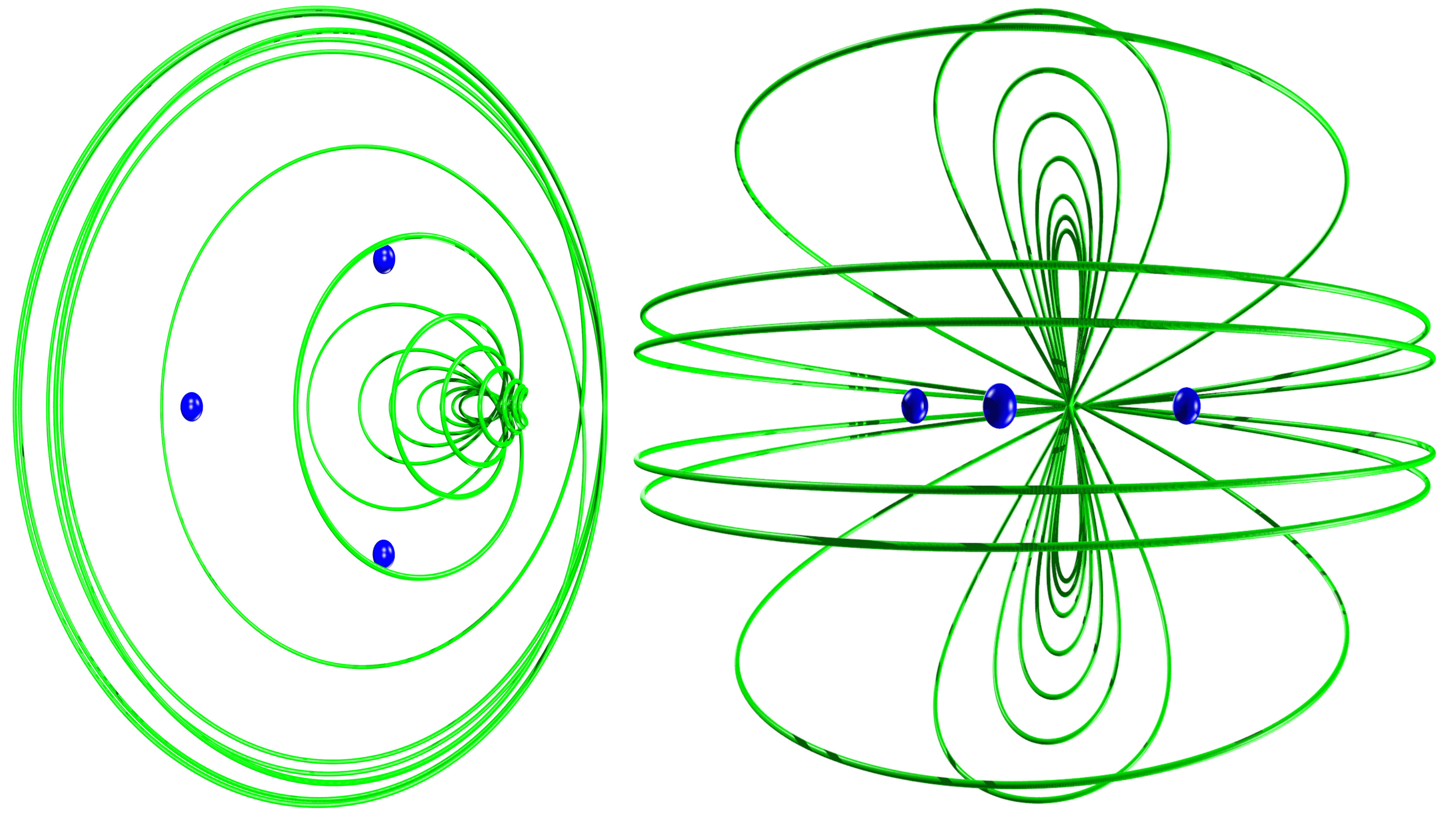} 
\caption{\textbf{Vertical Lyapunov family at $\mathcal{L}_5$:} spatial family of 
periodic orbits attached to $\mathcal{L}_5$ in the triple Copenhagen problem.
The ``tube'' of orbits is parameterized by energy/frequency, so that (locally)
the periodic orbits are isolated in the energy level.  Orbits near $\mathcal{L}_5$
inherit its stability, so that many of the orbits in the picture have complex
conjugate stable/unstable Floquet multipliers.
The families at $\mathcal{L}_{4,6}$ are obtained by $\pm 120$ degree rotations.}
\label{fig:vert_L5}
\end{figure}

\begin{figure}
\includegraphics[width=6.5in]{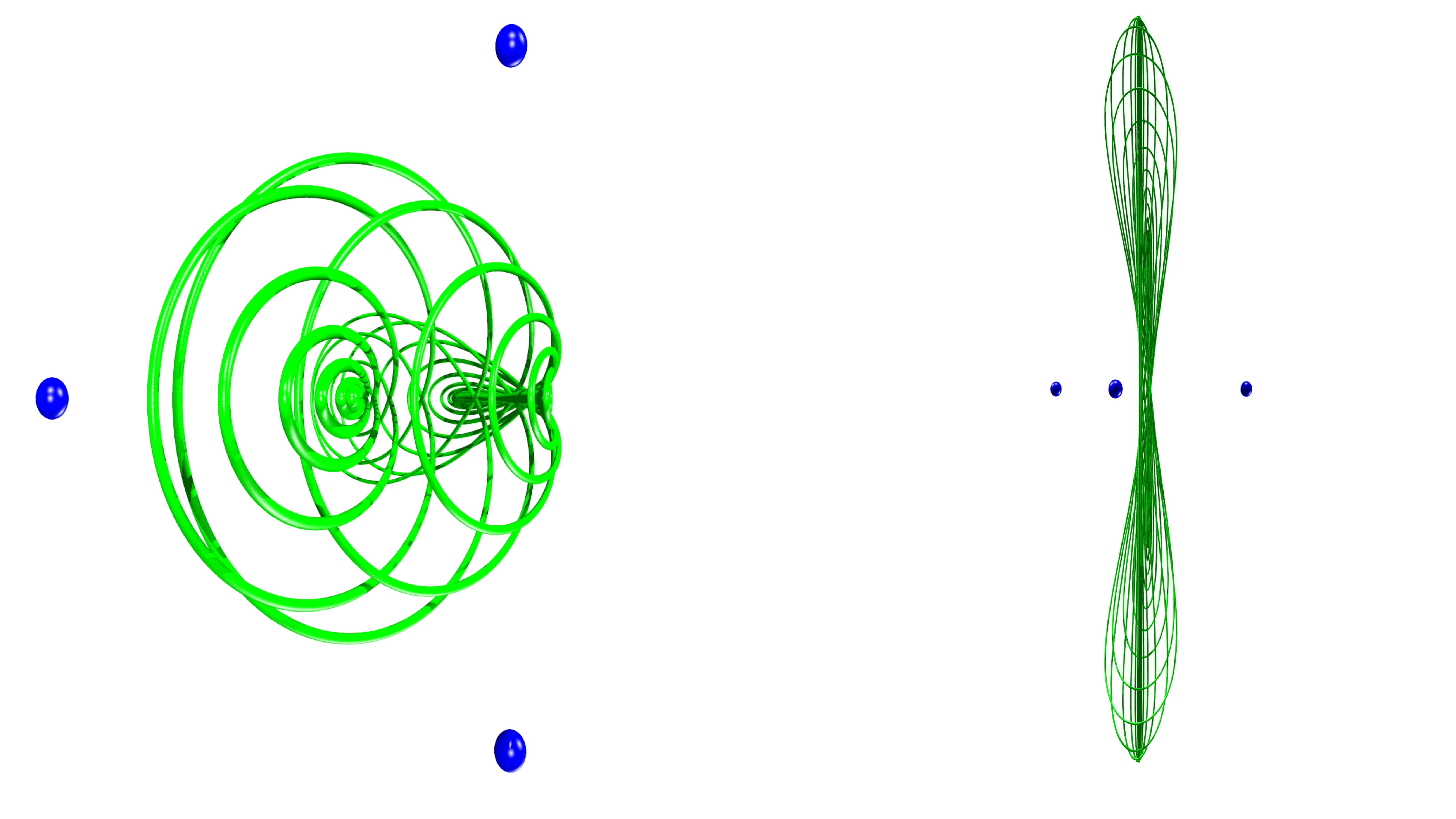} 
\caption{\textbf{Vertical Lyapunov families at $\mathcal{L}_0$ and $\mathcal{L}_2$:} 
spatial families of periodic orbits at $\mathcal{L}_{0,2}$.  The $\mathcal{L}_0$ family 
is coincident with the $z$-axis and the $\mathcal{L}_2$ family accumulates on 
the $\mathcal{L}_0$ family.  
Orbits near $\mathcal{L}_0$
inherit its stability, so that many of the orbits in the picture have complex
conjugate stable/unstable Floquet multipliers.
The $\mathcal{L}_{1,3}$ families are obtained from the 
$\mathcal{L}_2$ family by $\pm 120$ degree rotations.}
\label{fig:vert_L0}
\end{figure}

\subsection{Vertical Lyapunov Families} \label{sec:lyapFamilies}
As mentioned briefly above, the spatial CRFBP inherits the libration points of the planar problem.
 Moreover, the spatial problem has no out of plane
 equilibrium solutions.  In terms of stability, each planar equilibrium solutions 
picks up a center direction when embedded in the spatial problem.  That is, 
each of the spatial libration points 
has a purely imaginary pair of eigenvalues $\pm i \omega$ associated with an out of plane 
eigenspace.  The Lyapunov center theorem
\cite{MR0021186,MR2189486,MR96021}
is used to prove that there is a one parameter family of periodic orbits 
tangent to the vertical eigenspace of each libration point.
The family can be computed by numerical continuation begun in a small neighborhood 
of the libration point.  

The vertical family at $\mathcal{L}_5$ for the triple Copenhagen problem is illustrated in 
Figure \ref{fig:vert_L5}.  Initially the orbits have a ``figure eight'' shape, with the eight pinched
at the libration point and one lobe above and one below the $z= 0$ plane.  
The family is parameterized by energy and as energy is increased the eight opens up and 
eventually ``tips'', returning to the plane.  The union of the $\mathcal{L}_5$ family forms
a sphere in configuration space enclosing the three primaries.  
The vertical families at $\mathcal{L}_{4, 6}$ are obtained by $\pm 120$ degree rotations.

The situation at $\mathcal{L}_0$ in the triple Copenhagen problem is illustrated in 
Figure \ref{fig:vert_L0}.  Due to the symmetry of the problem the $\mathcal{L}_0$ vertical 
family moves entirely on the $z$-axis.  The $\mathcal{L}_2$ vertical family is illustrated 
in the same picture.  This family also appears to form a sphere, but in this case the orbits 
eventually accumulate on the $\mathcal{L}_0$ family on the $z$-axis.
The vertical families at $\mathcal{L}_{1,3}$ are obtained by $\pm 120$ degree 
rotation of the $\mathcal{L}_2$ family, and hence all three families accumulate at
$\mathcal{L}_0$.  

The existence of the spatial periodic orbits illustrated in Figures \ref{fig:vert_L5}
and \ref{fig:vert_L0}, along with many other such results, are 
proven with computer assistance in 
\cite{MR3896998}.
For small enough out of plane amplitude orbits in the vertical families inherit
their stability from the stability of the planar libration point.  Then at 
$\mathcal{L}_{0, 4,5,6}$ the vertical Lyapunov orbits have complex conjugate 
stable/unstable Floquet  exponents for some range of out of plane amplitudes.

%
%
%
 
		
\section{Parameterization of stable/unstable manifolds} \label{sec:parameterizationMethod}
In this section we review the parameterization method for stable/unstable manifolds attached to 
periodic solutions of ordinary differential equations, with an eye toward numerical calculations.
Much of the material has appeared in other places, and is included here for the benefit of the 
reader not familiar with these developments.  Indeed it is our hope that the present section 
provides a useful introduction to the ideas in context if a highly non-trivial application problem.
We also stress that a concrete description of the method 
\textit{accompanied by a complete description of the numerical implementation}
for a periodic orbit in a gravitational $N$-body problem having complex 
conjugate Floquet exponents -- and hence a three dimensional stable/unstable manifold-- 
has not appeared before.  Hence our little tutorial has some novelty.
However the reader familiar with this material may want to skip this section 
upon first reading.

\begin{figure}
\includegraphics[width=6.5in]{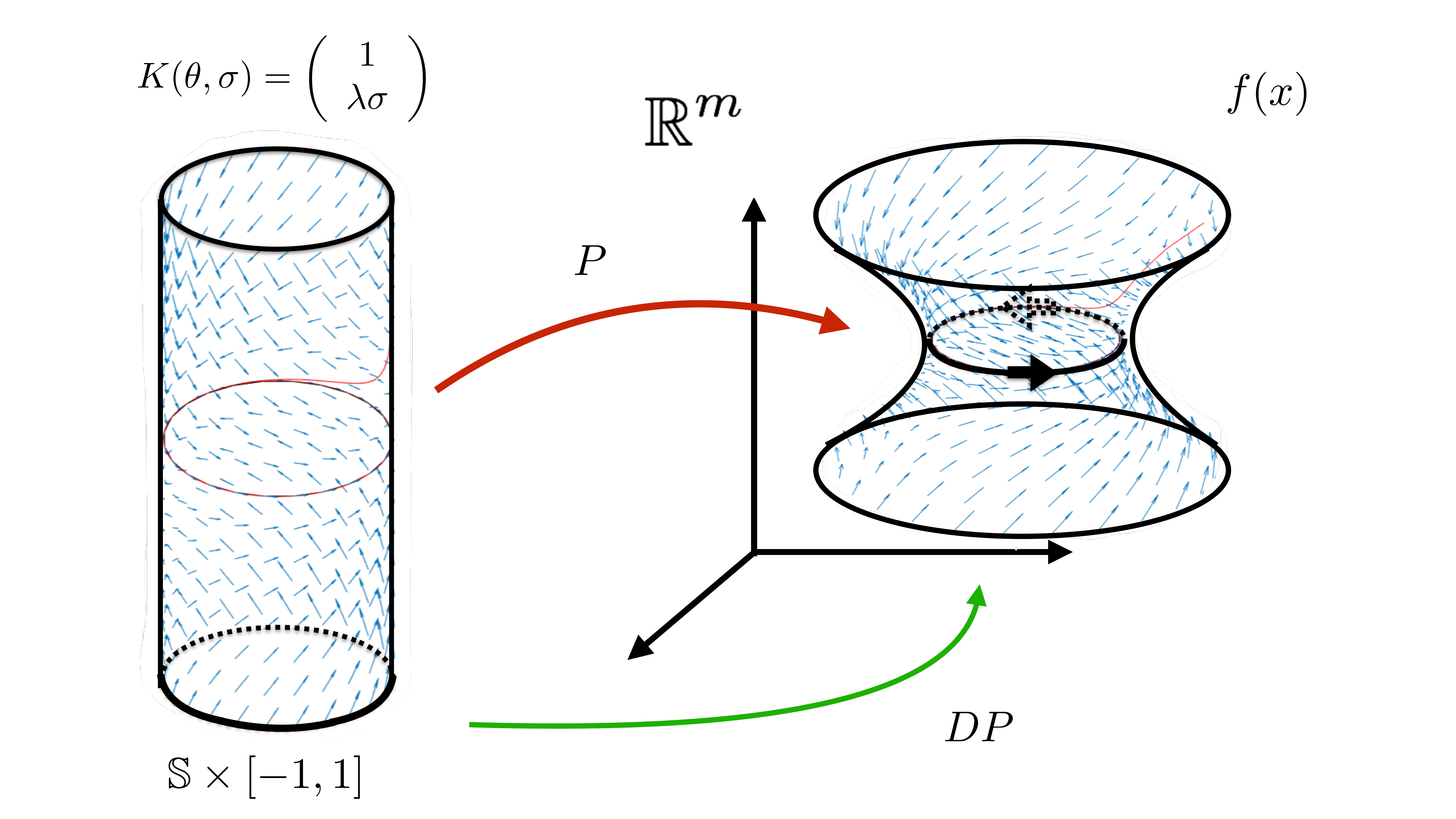} 
\caption{\textbf{Differential geometry of the parameterization method:} the 
geometric idea behind the parameterization method for vector fields is that 
a model vector field $K$, when pushed forward by the parameterization $P$, 
should match the given vector field $f$ on the image of $P$.  Under this assumption
the map $P$ takes orbits of $K$ to orbits of $f$ on the image of $P$.  Since the orbits 
of $K$ are known, we discover the dynamics on the image of $P$.  If $K$ models 
stable (respectively unstable) dynamics for a periodic orbit then $P$ parameterizes
a local stable (respectively unstable) manifold.  The relationship just described is 
quantified in the invariance Equation \eqref{eq:conjugacy}.}
\label{fig:parmMethod}
\end{figure}

The parameterization method is a functional analytic framework for studying invariant manifolds
which is useful in both theoretical and numerical settings.  The method has been successfully 
applied to the study of stable/unstable manifolds attached to 
fixed points of nonlinear maps between Banach spaces 
\cite{MR1976079,MR1976080}, invariant circles and their whiskers 
in quasi-periodically forced maps, 
stable/unstable manifolds of normally hyperbolic invariant tori
\cite{MR2240743,MR2289544,MR2299977}, 
stable/unstable manifolds attached to equilibrium and periodic 
solutions of parabolic PDEs \cite{jpRafPDE,parmPDE}, and 
quasi-periodic solutions and invariant tori in infinite dimensional 
systems \cite{MR2505176,MR2528494,rafaelStDelay_I,rafaelStDelay_II}.
The parameterization method has also been used to formulate 
KAM theorems without the use of action angle variables
\cite{MR2122688,MR2966749},
and dissipative KAM theorems 
\cite{MR3062760,MR3095277}.
Several works focusing on numerical aspects of the method are 
\cite{maximeJPMe,parmChristian,kotParm,fastSlow,MR3068557,MR3032844,MR2507323,jayChrisParmDDE,MR2851901}.

We make special mention of the works of 
\cite{MR2177465,MR2551254,MR3118249,doi:10.1137/140960207,poManProofs,chebManifolds}.
These papers deal with various aspects of the parameterization method 
for stable/unstable manifolds attached to periodic solutions of differential equations, 
and are the basis of the approach to invariant manifolds employed in the present 
work.  The rest of the section is devoted to the review of this material.  

\subsection{Invariance and homological equations}
Since the invariant manifolds we consider in the present work have complex 
Floquet exponents we describe the method in the context of complex vector fields.
So, consider the analytic vector field
$f:\mathbb{C}^m \to \mathbb{C}^m$ and the associated first order 
system of ordinary differential equations $\dot x = f(x)$.  Suppose that 
$\gamma:\mathbb{R} \to \mathbb{C}^m$ and $T > 0$ have
\[
\frac{d}{dt}\gamma(t) = f(\gamma(t)),
\]
and 
\[
\gamma(t + T) = \gamma(t), 
\]
for all $t \in \mathbb{R}$.
We say that $\gamma$ is a $T$-periodic solution of the differential equation.  

Assume now that $\gamma$ has $n$ stable (or unstable) Floquet exponents, which 
we denote by $\lambda_1,\lambda_2,\hdots,\lambda_n \in \mathbb{C}$.
Let $\mathbb{D}^n \subset \mathbb{C}^n$ denote the $n$-dimensional unit
poly-disk.
 Then it is natural to look for a parameterization 
$P \colon \mathbb{R} \times \mathbb{D}^n \to \mathbb{C}^m$  -- $T$ periodic in the first variable --
of the associated stable (or unstable) manifold.  

Let $\Lambda$ denote the $n\times n$ diagonal matrix of stable exponents.
Then an appropriate model of the stable manifold is the cylinder 
$\mathbb{S}^1 \times \mathbb{D}^n$ endowed with the linear vector field
\[
K(\theta, \sigma) = \left(
\begin{array}{c}
1 \\
\Lambda \sigma 
\end{array}
\right).
\]
Observe that this field has a $1$-periodic solution at $\sigma = 0$, and that all orbits 
with $\sigma \neq 0$ converge exponentially to this periodic orbit, making this the 
simplest possible model for the dynamics on the stable manifold.

The geometric idea behind the parameterization method is to look for a parameterization $P$ satisfying the 
infinitesimal conjugacy
\[
DP(\theta, \sigma) K(\theta, \sigma) = f(P(\theta, \sigma)).
\]
The equation demands that the push forward of the vector field $K$
by $P$ is equal to the vector field $f$ restricted to the image of $P$.
If the vector fields are equal then they generate the same dynamics (same orbits).
But the orbits of $K$ are known explicitly, and we have that any such $P$
parameterizes a local stable manifold for $\gamma$.
The situation is illustrated schematically in Figure \ref{fig:parmMethod}.
Expanding the first order differential operator $DP \circ K$ on the left leads to 
the invariance equation
\begin{equation}\label{eq:conjugacy}
\frac{d}{d\theta}P(\theta,\sigma) +
\sum_{i=1}^n \lambda_i\sigma_i \frac{\partial}{\partial \sigma_i} P(\theta,\sigma) 
= f(P(\theta,\sigma)),
\end{equation}
which is a first order system of PDEs for $P$.
We impose the first order constraints
\begin{equation}\label{eq:ConditionA0}
P(\theta,0)= \gamma(\theta)
\end{equation}
and
\[
\frac{\partial}{\partial\sigma_i} P(\theta,0) = v_i(\theta)
\]
for $i=1,2,\hdots,n$ where $v_i(\theta)$ -- the stable (or unstable) normal bundle associated 
with the Floquet exponent $\lambda_i$ -- solves the linear differential equation
\begin{equation} \label{eq:bundleEqn}
-v_j(t)' + Df(\gamma(t)) v_j(t) =  \lambda_j v_j(t), 
\end{equation}
for each $1 \leq j \leq n$.  The function $v_j$ is either $T$ periodic or $2 T$ periodic 
depending on wether the associated bundle is orientable or not.

%
%

Let $\Phi \colon \mathbb{C}^m \times  \mathbb{R} \to \mathbb{C}^m$ be the flow generated by $f$.
It can be shown (see any of the references given at the end of the last section) that if $P$
is a solution of the infinitesimal invariance Equation \eqref{eq:conjugacy}, then $P$ satisfies 
the flow conjugacy 
\begin{equation} \label{eq:flowConj}
\Phi(P(\theta, \sigma), t) = P\left(\theta + t, e^{\Lambda t} \sigma \right), 
\end{equation}
for all $t \geq 0$ and $\sigma \in \mathbb{D}^n$.
Then in fact the parameterization method recovers the dynamics on the 
manifold in addition to the embedding.

Since $f$ is analytic we look for an analytic $P$.  To this end suppose that $P$
has the power series expansion

\begin{equation}\label{eq:Ptaylorexpansion}
P(\theta,\sigma)= \sum_{|\alpha|=0}^\infty A_\alpha(\theta) \sigma^\alpha, 
\end{equation}
where for each multi-index $\alpha= (\alpha_1,\alpha_2,\hdots,\alpha_n) \in \mathbb{N}^\alpha$,
\[
|\alpha| = \alpha_1 +\alpha_2 + \hdots + \alpha_n
\] 
and given $\sigma \in \mathbb{D}^n$ we denote
\[
\sigma^\alpha = \sigma_1^{\alpha_1}\cdot \sigma_2^{\alpha_2}\cdot \hdots \cdot \sigma_n^{\alpha_n}.
\]
Moreover, each of the coefficients $A_\alpha \colon \mathbb{R} \to \mathbb{C}^m$ are $\mathbb{T}-$
periodic complex functions. Plugging the expansion \eqref{eq:Ptaylorexpansion} 
in \eqref{eq:conjugacy} and matching power of $\sigma$ leads to the 
\textit{homological equation} for  $A_\alpha(t)$ given by 
\begin{equation}\label{eq:conjugacyalpha}
\frac{d}{d\theta}A_\alpha(\theta) +\langle \alpha, \Lambda \rangle A_\alpha(\theta) = f(P(\theta,\sigma))_\alpha.
\end{equation}
where 
\[
\langle \alpha, \Lambda \rangle = \alpha_1\lambda_1+\hdots+\alpha_n\lambda_n.
\]

\begin{remark}[Non resonance criteria] \label{rem:nonRes}
{\em
We say that the Floquet exponents 
$\lambda_1,\hdots,\lambda_n$ are resonant at order 
$k$ if there exist $\alpha \in \mathbb{N}^n$ so that 
$|\alpha| = k$ and 
\[
 \langle \alpha,\Lambda \rangle = \lambda_i
\]
for some $\lambda_i$ a Floquet exponent of $\gamma$.
We recall (see again any of the references cited in the last paragraph of the previous 
section) that the homological equations are uniquely solvable to all orders if
and only if there are no resonances for $|\alpha|\geq 2$. 
In this case we say that the Floquet exponents are non-resonant.  
For the examples considered in the remainder of the paper
the periodic orbits had a single complex conjugate pair of stable/unstable 
Floquet exponents and all the other exponents are purely imaginary.  
In such a case there is no possibility of resonances for $|\alpha| \geq 2$, 
and the parameterization coefficients $A_\alpha(\theta)$ are guaranteed to 
be defined to all orders.}
\end{remark}

Since solutions of the homological equations
are $T$ periodic for all $\alpha$, it is natural to expand it using Fourier series. 
Letting $\omega = \frac{2\pi}{T}$ where $T$ is the period of $\gamma$
we look for $A_\alpha$ expressed as
\[
A_\alpha(\theta) = \sum_{k\in \mathbb{Z}} a_{\alpha,k} e^{\im \omega k \theta}.
\]
Then, one can plug the expansion in \eqref{eq:conjugacyalpha} and rewrite the problem 
as the zero of a nonlinear operator defined on the space of Fourier coefficients. 
The focus of next section is to solve \eqref{eq:conjugacyalpha} up to some finite 
order using a finite dimensional Fourier expansion. The process will be explicitly presented 
in the case of the CRFBP.

\begin{figure}[t!]
\subfigure{\includegraphics[width=0.6\textwidth]{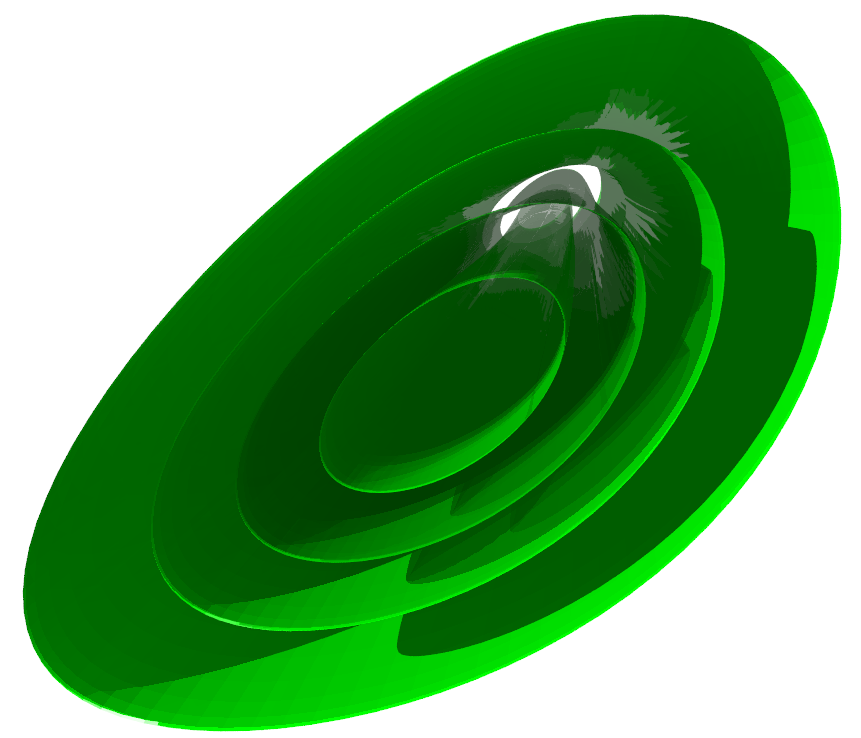}}
\subfigure{\includegraphics[width=0.2\textwidth]{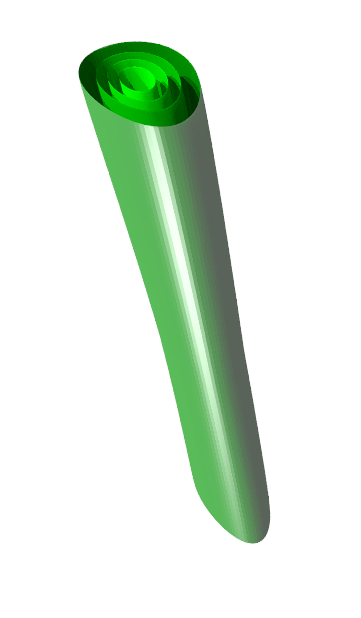}}
\caption{\textbf{Parameterization of the local stable manifold of a vertical Lyapunov orbit at $\mathcal{L}_0$:} 
(left) top and (right) side view when $m_1=0.4$ and $m_2=0.35$. 
The parameterization is computed to Taylor order $5$ with $20$ Fourier nodes per Taylor coefficient. 
The image displays the boundary torus of the parameterization $P(\theta, \sigma, \overline\sigma)$ where 
where $\| \sigma \| = R$ for $R =  0.3, 0.5, 0.7,1$. We remark that in each case the torus is very thin, so that 
in the image each torus looks essentially like a cylinder.
The largest torus in the present Figure is roughly the same size as the one shown in Figure \ref{fig:L0unstable}, 
however that figure illustrates the unstable parameterization. This nevertheless gives a sense of the scale of 
the local parameterizations, namely where are the primaries located.}\label{fig:ManifoldExample}
\end{figure}

\begin{figure}
\subfigure{\includegraphics[width=0.7\textwidth]{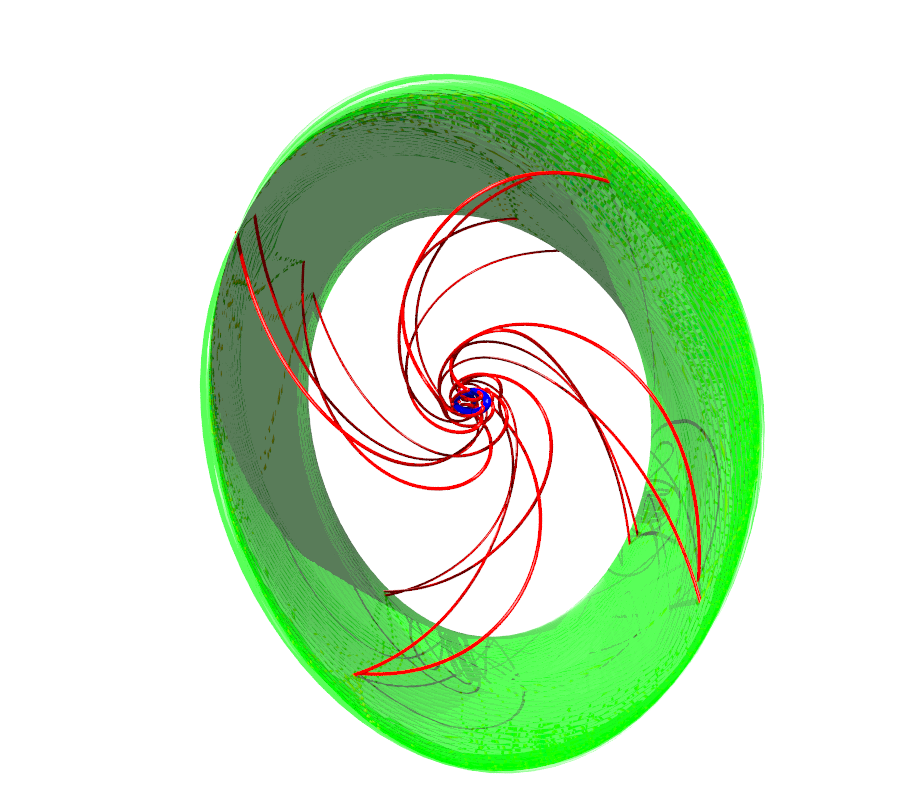} } \\
\subfigure{\includegraphics[width=0.7\textwidth]{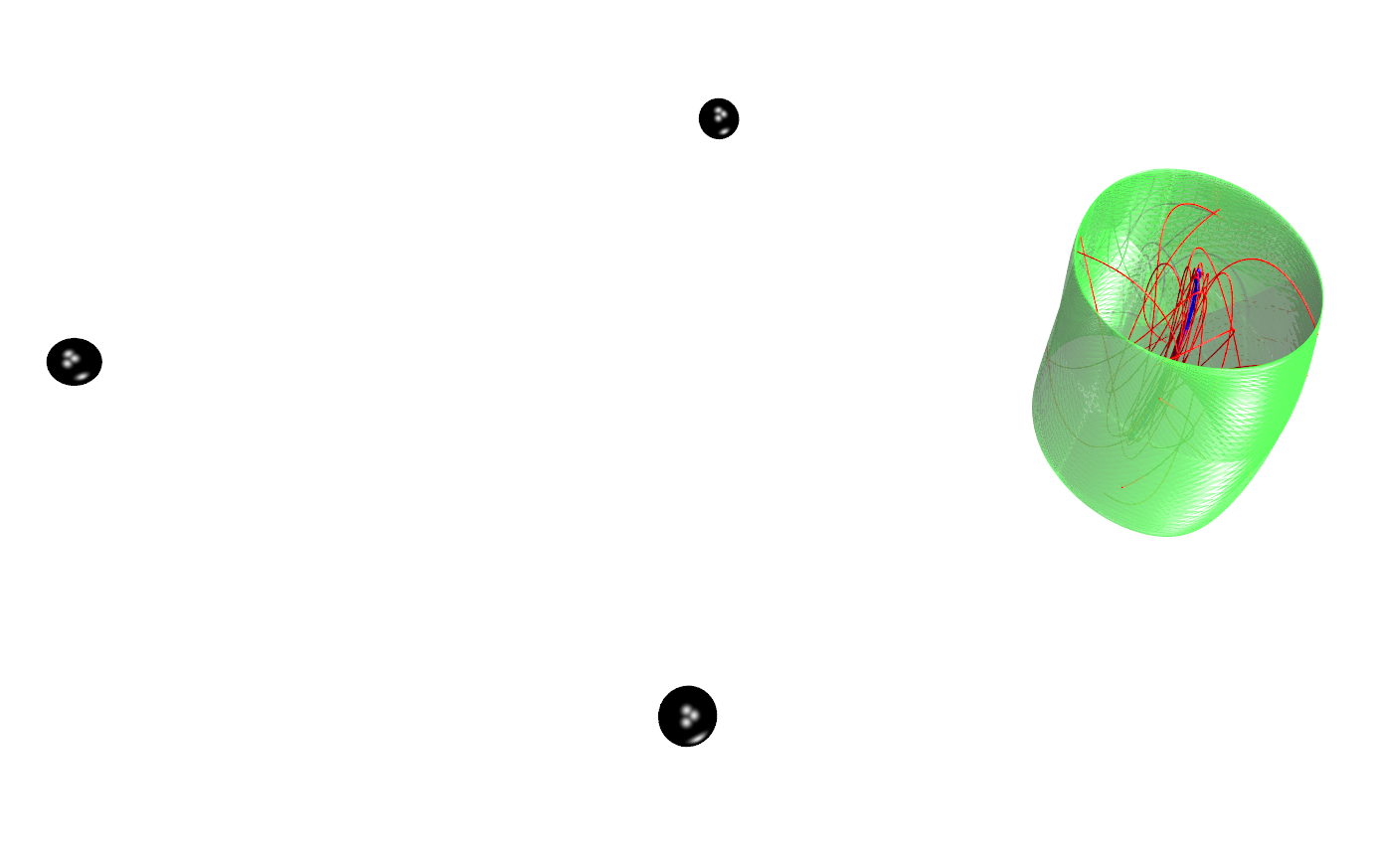} }
\caption{\textbf{Example:}
\textbf{Parameterization of the stable manifold of a vertical Lyapunov orbit at $\mathcal{L}_5$:}
 (top frame) top and (bottom frame) side views in the case of  equal masses. 
 The boundary torus of the parameterized stable manifold is displayed in green, while the periodic orbit itself 
is in blue (torus is very thin). 
We use the conjugacy relation to generate sixteen forward asymptotic trajectories 
with initial data on the boundary, 
giving a sense of the dynamics on the three dimensional stable manifold. The initial value are equally distributed
and the resulting trajectories are displayed in red. The manifold was computed with 
$20$ Fourier modes per Taylor coefficient, taking the Taylor expansion to polynomial order $5$. 
Observe that the image of the parameterization is ``macroscopic'' -- i.e. its size is of the same order
as the sides of the equilateral triangle.  
The same local manifold parameterization 
is used to find homoclinic connections for the periodic orbit, see Figure \ref{fig:homoclinicL5}. }
\end{figure}

\begin{figure}
\subfigure{\includegraphics[width=0.7\textwidth]{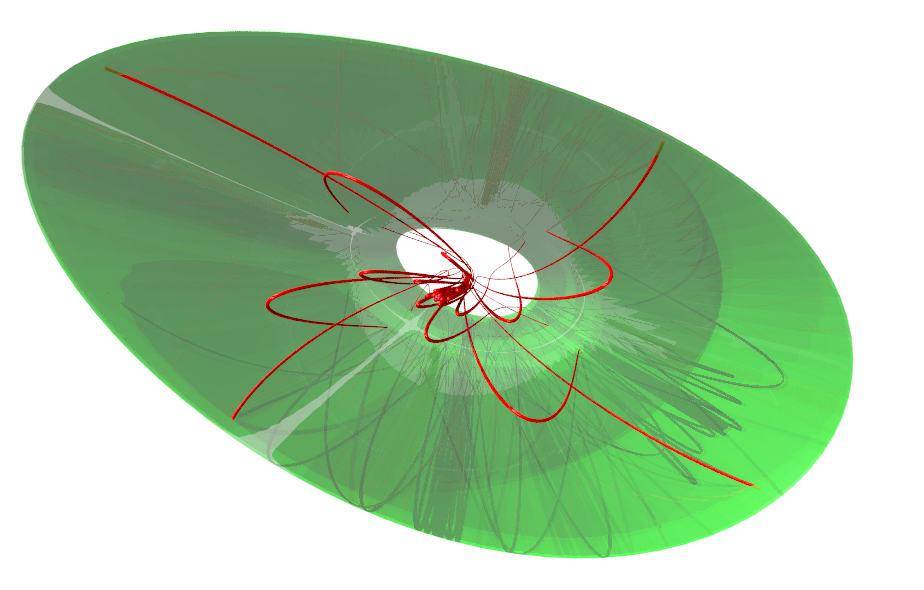} }\\
\subfigure{\includegraphics[width=0.7\textwidth]{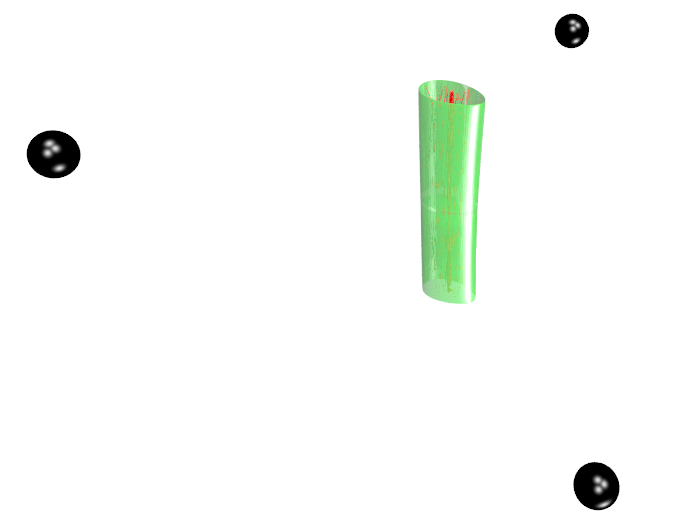} }
\caption{\textbf{Example:}
\textbf{Parameterization of the unstable manifold of a vertical Lyapunov orbit at $\mathcal{L}_0$:}
Top and side view of the local unstable manifold attached to a vertical Lyapunov orbit at $\mathcal{L}_0$ in the CRTBP
with $m_1=0.4$ and $m_2= 0.35$. The boundary torus of the parameterized manifold is displayed in green
(torus is very thin). 
We use the conjugacy relation to simulate forward trajectory for initial data on the boundary,  giving a sense of the
 dynamics on the manifold. The initial value are equally distributed on the domain of the parameterization and the 
 resulting trajectories are displayed in red. The manifold was computed with $20$ Fourier modes per Taylor coefficient, 
 taking the Taylor expansion to polynomial order $5$. The periodic orbit itself is not visible but we note that every trajectory 
 in red accumulates to the orbit in backward time. The same local manifold parameterization is used to find homoclinic 
 connections for the periodic orbit, see Figure \ref{fig:L0spacial315}. }\label{fig:L0unstable}
\end{figure}

\subsection{Parameterized manifolds in the CRFBP}\label{sec:Manifold}

Our goal is to solve Equation \eqref{eq:conjugacyalpha} for the CRFBP.
In fact, we use the the idea discussed in Appendix \ref{sec:autoDiff} and first 
pass to an equivalent polynomial vector field $f \colon \mathbb{R}^9 \to \mathbb{R}^9$. 
Having polynomial nonlinearities greatly simplifies 
the formal series calculations, as Fourier-Taylor series are multiplied as 
follows.

Suppose that $g, h: [0,T] \times \mathbb{D}^n \to \mathbb{C}$ are given by 
\[
g(t,\sigma)= \sum_{|\alpha|=0}^\infty   \sum_{k\in \mathbb{Z}} a_{\alpha,k}e^{\im k \omega t} \sigma^\alpha,
\]
and 
\[
h(t,\sigma)= \sum_{|\alpha|=0}^\infty  \sum_{k\in \mathbb{Z}} b_{\alpha,k}e^{\im k \omega t} \sigma^\alpha.
\]
We refer to $g$ and $h$ as Fourier-Taylor series and to 
\[
a = \{ a_{\alpha,k} \in \mathbb{C} : \alpha \in \mathbb{N}^n ~\mbox{and}~ k \in \mathbb{Z} \},
\]
and
\[
b = \{ b_{\alpha,k} \in \mathbb{C} : \alpha \in \mathbb{N}^n ~\mbox{and}~ k \in \mathbb{Z} \},
\] 
as the Fourier-Taylor coefficients of $g$ and $h$ respectively.  Observe that $g, h$ are $T$ periodic in $t$.

\begin{definition}[Convolution product]\label{def:Coefficients}
The Fourier-Taylor series of the point-wise product $g\cdot h (t,\sigma)$ is
\[
(g\cdot h) (t,\sigma) =  \sum_{|\alpha|=0}^\infty   \sum_{k\in \mathbb{Z}}(a \ast b)_{\alpha,k}e^{\im k \omega t} \sigma^\alpha,
\]
where the Fourier-Taylor coefficients are given by the Cauchy-convolution products 
\[
(a \ast b)_{\alpha,k} =  \sum_{\substack{ \alpha_1 +\alpha_2 =\alpha \\ \alpha_1,\alpha_2 \in \mathbb{N}^n}} 
\sum_{\substack{ k_1 +k_2 =k \\ k_1,k_2 \in \mathbb{Z}}}
a_{\alpha_1,k_1}\cdot b_{\alpha_2,k_2}.
\]
We refer to $\ast$ as the Cauchy-convolution product of $a$ and $b$.  
\end{definition}  

The definition extends also to higher order powers.    For example 
\[
g^3(t, \sigma) = \sum_{|\alpha| = 0}^\infty \sum_{k \in \mathbb{Z}} (a * a * a)_{\alpha, k} e^{i k \omega t} \sigma^{\alpha}, 
\]
where 
\[
(a * a * a)_{\alpha, k} =  \sum_{\substack{ \alpha_1 +\alpha_2 + \alpha_3 =\alpha \\ \alpha_1,\alpha_2, \alpha_3 \in \mathbb{N}^n}} 
\sum_{\substack{ k_1 +k_2 + k_3=k \\ k_1,k_2, k_3 \in \mathbb{Z}}}
a_{\alpha_1,k_1}\cdot a_{\alpha_2,k_2}  \cdot a_{\alpha_3, k_3}.
\]
Quartic and quintic powers are defined in the analogous way.

We now look for the Fourier-Taylor coefficients of the stable (or unstable)
manifold parameterization, which we write as
\[
P(\theta, \sigma) = \sum_{|\alpha| = 0}^\infty \sum_{k \in \mathbb{Z}} a_{\alpha, k} e^{i \omega k \theta} \sigma^\alpha, 
\]
where 
\[
a_{\alpha, k} = 
\left(
\begin{array}{c}
a^1_{\alpha, k} \\
a^2_{\alpha, k} \\
a^3_{\alpha, k} \\
a^4_{\alpha, k} \\
a^5_{\alpha, k} \\
a^6_{\alpha, k} \\
a^7_{\alpha, k} \\
a^8_{\alpha, k} \\
a^9_{\alpha, k} 
\end{array}
\right) \in \mathbb{C}^9,
\]
for each $\alpha \in \mathbb{N}^n$ and $k \in \mathbb{Z}$.
Observe that $\{a_{0, k}\}_{k \in \mathbb{Z}}$ 
and $\{a_{e_j, k}\}_{k \in \mathbb{Z}}$ are the Fourier coefficients of the periodic orbit
and the $j$-th normal bundle respectively.   

After rewriting the CRFBP as a polynomial system (see again  Appendix \ref{sec:autoDiff})
and projecting Equation \eqref{eq:conjugacyalpha} for the resulting polynomial field 
into Fourier-Taylor coefficient space, we obtain for each $|\alpha| \geq 2$ an equivalent 
$F_\alpha(a) = 0$ problem, where $F_\alpha$ is the map given by 
\begin{align*}
F_{\alpha,k}^{1}(a)=  &(\im\omega k +\langle \alpha,\lambda \rangle)a_{\alpha,k}^{1} -a_{\alpha,k}^{2}, \\
F_{\alpha,k}^{2}(a)=  & (\im\omega k +\langle \alpha,\lambda \rangle)a_{\alpha,k}^{2} -2a_{\alpha,k}^{4} -a_{\alpha,k}^{1} +\sum_{i=1}^3 m_i\left((a^{1}-x_i)\ast a^{6+i} \ast a^{6+i}\ast a^{6+i}\right)_{\alpha,k},  \\ 
F_{\alpha,k}^{3}(a)=  &(\im\omega k +\langle \alpha,\lambda \rangle)a_{\alpha,k}^{3} -a_{\alpha,k}^{4}, \\
F_{\alpha,k}^{4}(a)= & (\im\omega k +\langle \alpha,\lambda \rangle)a_{\alpha,k}^{4} +2a_{\alpha,k}^{2} -a_{\alpha,k}^{3} +\sum_{i=1}^3 m_i\left((a^{3}-y_i)\ast a^{6+i}\ast a^{6+i}\ast a^{6+i}\right)_{\alpha,k}, \\
F_{\alpha,k}^{5}(a)= &(\im\omega k +\langle \alpha,\lambda \rangle)a_{\alpha,k}^{5} -a_{\alpha,k}^{6},  \\
F_{\alpha,k}^{6}(a)= &(\im\omega k +\langle \alpha,\lambda \rangle)a_{\alpha,k}^{6} + \sum_{i=1}^3 m_i\left((a^{5}-z_i)\ast a^{6+i}\ast a^{6+i} \ast a^{6+i}\right)_{\alpha,k}, \\
F_{\alpha,k}^{7}(a)= & (\im\omega k +\langle \alpha,\lambda \rangle)a_{\alpha,k}^{7} +\left( \left((a^{1}-x_1)\ast a^2 +(a^3-y_1)\ast a^4 +(a^5-z_1)\ast a^6\right) \ast a^7 \ast a^7 \ast a^7 \right)_{\alpha,k}, \\
F_{\alpha,k}^{8}(a)= & (\im\omega k +\langle \alpha,\lambda \rangle)a_{\alpha,k}^{8} +\left( \left((a^{1}-x_2)\ast a^2 +(a^3-y_2)\ast a^4 +(a^5-z_2)\ast a^6\right)\ast a^8 \ast a^8 \ast a^8 \right)_{\alpha,k}, \\
F_{\alpha,k}^{9}(a)= & (\im\omega k +\langle \alpha,\lambda \rangle)a_{\alpha,k}^{9} +\left( \left((a^{1}-x_3)\ast a^2 +(a^3-y_3)\ast a^4 +(a^5-z_3)\ast a^6\right) \ast a^9 \ast a^9 \ast a^9 \right)_{\alpha,k},\\
\end{align*}
Here $(x_i,y_i,z_i)$ for $i = 1,2,3$ denote the coordinates of the primaries.

Choose a Taylor truncation order $N \geq 2$ and recursively solve the equations $F_{\alpha}(A) = 0$ for each $2 \leq |\alpha| \leq N$
using Newton's method.  Let $\left\{\overline{a}_{\alpha,k}^j\right\}_{|\alpha \leq N, |k| < K}$, $1 \leq j \leq 9$ denote the resulting 
numerically computed approximate solutions. This results in the polynomial approximation 
\[
P^{(N,K)}(\theta,\sigma) = \sum_{|\alpha|=0}^N \sum_{|k|<K} \bar{a}_{\alpha,k}e^{\im \omega k \theta}\sigma^\alpha
\]
of the desired stable (unstable) manifold parameterization.

%

\begin{remark}[Symmetry in the case of complex conjugate eigenvalues]\label{remark:symmetry}
{\em We are interested in the case $n=2$ with $\lambda_{1,2}=a\pm \im b$. The parameterization will 
have a complex coefficients/image, however one can check that the Taylor coefficients have the symmetry
\[
A_{\alpha_1, \alpha_2}(\theta) = \overline{A_{\alpha_2,\alpha_1}(\theta)},~~\forall (\alpha_1, \alpha_2) \in \mathbb{N}^2.
\]
Indeed this follows directly from the complex conjugate symmetry of Equation  \eqref{eq:conjugacyalpha}.
So, for $\sigma= (\sigma_1,\sigma_2) \in \mathbb{R}^2$ we define 
$\hat{P}(\theta,\sigma)= P(\theta,\sigma_1+\im \sigma_2,\sigma_1-\im\sigma_2)$ and have that 
the image of $\hat{P}$ is real thanks to the symmetry above. 
Since we are studying a real vector field we are ultimately interested in only real image of the 
parameterization, and in future applications of the method we 
always use the complex conjugate variables just discussed.
We also note that the symmetry is further inherited by the Fourier coefficients. 
That is, for all $k\in \mathbb{Z}$ and for any multi-index $\alpha=(\alpha_1,\alpha_2)$ 
we define $\beta= (\alpha_2,\alpha_1)$.  It follows that
\[
a_{\alpha,k} = \bar{a}_{\beta,-k}.
\]
In particular the coefficients are real when $k=0$. One can use this fact to reduce the 
computation time as it follows that one needs only to compute half of the coefficients 
to determine the parameterization.}
\end{remark}

\subsection{Numerical examples}
We now return to the vertical Lyapunov families of periodic orbits at $\mathcal{L}_0$ and $\mathcal{L}_{4,5,6}$,
which for small out of plane amplitudes are insured to have complex conjugate Floquet exponents.
Indeed, we find that the orbits have the desired stability for fairly substantial out of plane amplitudes as well,
see the tables in Appendix B. For example, 
Figure \ref{fig:ManifoldExample} illustrates a periodic orbit at $\mathcal{L}_0$ with non-zero Floquet exponents of approximately  
$\pm 1.2744 \pm 0.8356\im$, so that it is possible to compute a three dimension manifold attached to the orbit. This manifold satisfies 
the symmetries previously stated and we focus on its real image. 
To simplify the MATLAB codes, we did not exploit the symmetries of the problem to reduce the dimension and thus solved the 
homological equations for all $\alpha$ up to order $5$.   This results in an approximate parameterization with 
$7,371$ non-zero Fourier-Taylor coefficients. To test the accuracy of the approximation we exploit
 the conjugacy relation as follows. We use numerical integration to evaluate
\[
E(P^{(N,K)},\theta_0,\sigma_0,t)= \left\| \Phi(P^{(N,K)}(\theta_0, \sigma_0), t) 
- P^{(N,K)}\left(\theta_0 + t, e^{\Lambda t} \sigma_0 \right) \right\|,
\]
where we recall that
\[
\Lambda= \begin{pmatrix}
\lambda_1 & 0 \\ 0 & \lambda_2 \\
\end{pmatrix}.
\]
To obtain best possible accuracy, we first fix the scale of the eigenvector and then choose the Taylor order so that the 
last coefficients have norm close to machine precision. For the manifold previously described and displayed in Figure 
\ref{fig:L0spacial}, we sample points on the boundary of the parameterization and approximate the error $E$ at those 
points using various integration time. We take initial values evenly distributed on the boundary of the domain of the 
parameterization, writing $(\theta,e^{\im \sigma},e^{-\im\sigma})$ with $(\theta,\sigma) \in [0,\tau] \times [0,2\pi]$. 
We approximated the error with the given stable manifold for $100$ points in this domain
 and obtained the following error approximation
\begin{align*}
\max_{1\leq i \leq 100} E(P^{(5,20)},\theta_i,\sigma_i,10^{-10}) &= 9.6467\cdot 10^{-11}, \\
\max_{1\leq i \leq 100} E(P^{(5,20)},\theta_i,\sigma_i,10^{-8}) &= 9.6475\cdot 10^{-11}, \\
\max_{1\leq i \leq 100} E(P^{(5,20)},\theta_i,\sigma_i,10^{-6}) &= 9.7219\cdot 10^{-11}, \\
\max_{1\leq i \leq 100} E(P^{(5,20)},\theta_i,\sigma_i,10^{-4}) &= 2.3987\cdot 10^{-9}, \\
\max_{1\leq i \leq 100} E(P^{(5,20)},\theta_i,\sigma_i,10^{-2}) &= 2.3055\cdot 10^{-7}. \\
\end{align*} 

See also Figures \ref{fig:homoclinicL5} and \ref{fig:L0unstable}
for other graphical illustrations of the results obtained using the parameterization method
for vertical Lyapunov orbits in the CRFBP.

\begin{figure}[t!]
\includegraphics[width=5.0in]{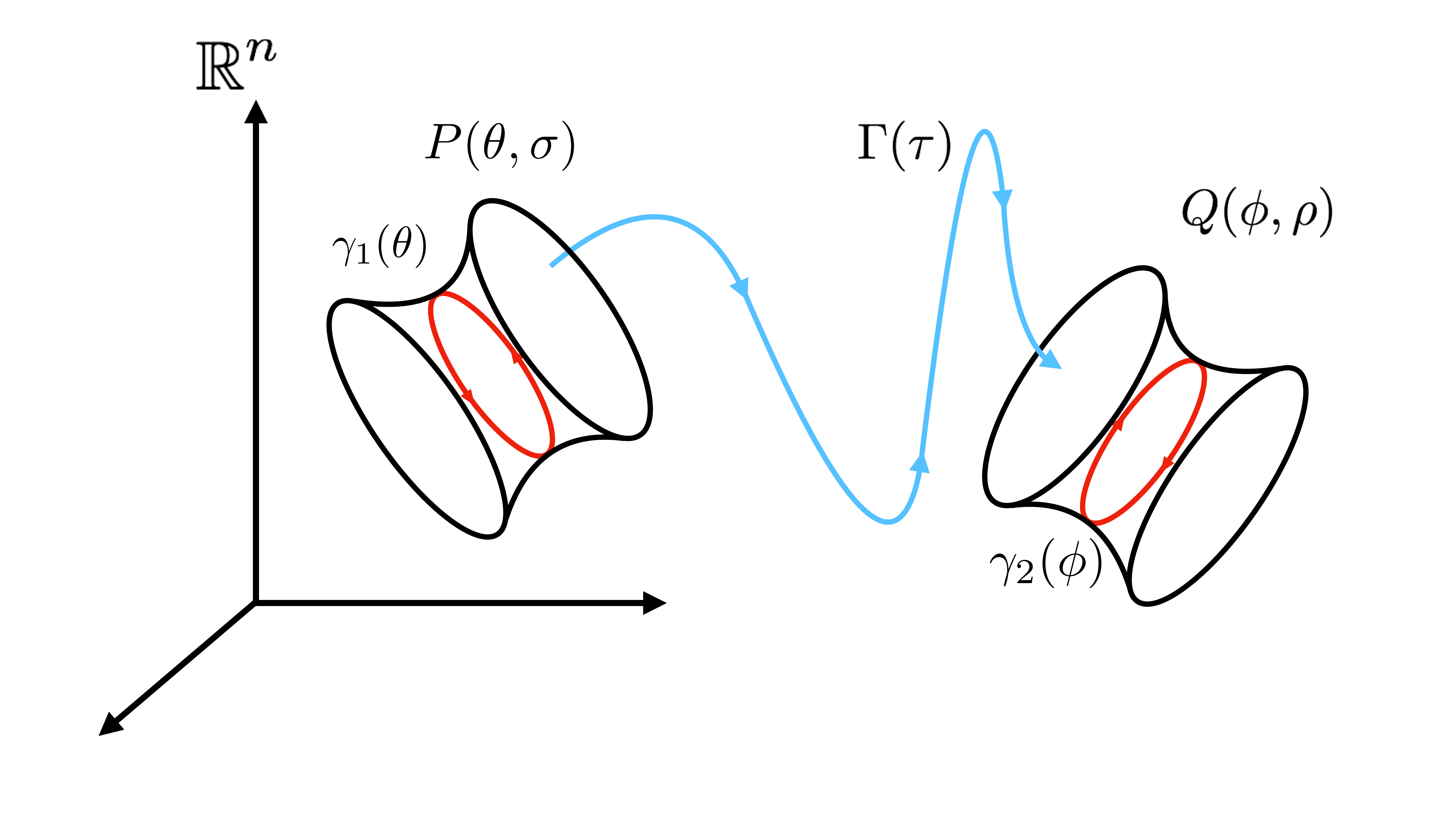} 
\caption{\textbf{Cycle-to-cycle connection:} $\gamma_1$ and $\gamma_2$ are periodic 
orbits with $P$ and $Q$ parameterizations of their local unstable and stable manifolds
respectively.  A homoclinic connection is equivalent to an orbit segment $\Gamma(\tau)$
beginning on the image of $P$ and terminating after time $T$ on the image of $Q$.  The equivalence is 
formalized as a two point boundary value problem in Equation \eqref{eq:bvp_1}. }
\label{fig:connectionBVP}
\end{figure}

\section{Cycle-to-cycle connections} \label{sec:connections}
To find the connection, we use the stable and unstable manifold parameterizations developed in the 
previous section to formulate a two point boundary value problem for a heteroclinic/homoclinic connecting orbit
asymptotic to a periodic solution of the spatial CRFBP. In the applications we consider the periodic orbit will be a member of one of the vertical Lyapunov 
families discussed in Section \ref{sec:lyapFamilies}.

For a connection to exist, the manifolds do not need to intersect transversely
in the full phase space, but rather in the energy manifold. 
Recall that the spatial CRFBP conserves the Jacobi integral, so that a trajectory $u(t)\in\rr^6$ 
solving $\dot u = f(u)$, with $f$ as in Equation \eqref{eq:FirstOrderOriginal},
must lie in a level set of the function
\[
J(u) = u_1^2 +u_3^2 +2\left(\frac{m_1}{r_1(u)} +\frac{m_2}{r_2(u)} +\frac{m_3}{r_3(u)}\right)  -(u_2^2 +u_4^2 +u_6^2).
\]
So, for a given periodic orbit $\gamma(t)$ there is a $K \in \mathbb{R}$ so that $K= J(\gamma(t))$ for all $t$. 
In fact we can find $K$ by choosing any $t_0 \in \mathbb{R}$ and evaluating 
\[
J(\gamma(t_0)) = K.
\]
Define
\[
\mathcal{K} = \left\{ u \in \mathbb{R}^6 : J(u)= K \right\},
\]
and note that $\mathcal{K}$ is locally a five dimensional manifold.

Consider the case where $\gamma$ has two stable and two unstable Floquet exponents, so that 
$W^{s,u}(\gamma)$ are three dimensional invariant manifolds.
From the continuity of $J$ it follows that 
$W^s(\gamma),W^u(\gamma) \subset \mathcal{K}$. Since $\mathcal{K}$ is five dimensional it is possible 
that a pair of three dimensional submanifolds can intersect transversely relative to $\mathcal{K}$.
It is highly unlikely that the images of the local stable/unstable manifold parameterizations intersect
except at $\gamma$, and it is necessary to look for a point on the local unstable manifold which is on 
local stable manifold at some later time.

To formalize the discussion let $\gamma_1, \gamma_2 \colon \mathbb{R} \to \mathbb{R}^6$ be periodic 
orbits with periods $T_1, T_2 > 0$ respectively.  Suppose that $J(\gamma_1(t)) = J(\gamma_2(t))$ 
(note that this condition is automatically satisfied if $\gamma_1 = \gamma_2$ -- the case of a 
homoclinic connection). 
Let $P, Q\colon \mathbb{R} \times B \to \mathbb{R}^6$ denote local unstable and stable manifold 
parameterizations respectively, where $B$ is the unit disk in the plane.
We seek $T > 0$, $\theta_0, \phi_0 \in \mathbb{R}$, $\sigma_0, \rho_0 \in B$,
and a function $\Gamma \colon [0, T] \to \mathbb{R}^6$ so that 
\begin{align}\label{eq:bvp_1}
\begin{cases}
\dot \Gamma (t) = f(\Gamma(t)), & \forall t \in (0,T) \\
\Gamma(0)= P(\theta_0,\sigma_0), &\theta_0, \in [0, T_1],  \sigma_0 \in B \\
\Gamma(T)= Q(\phi_0,\rho_0), & \phi_0 \in [0, T_2], \rho_0 \in B.
\end{cases}
\end{align}
That is , we seek an orbit segment $\Gamma$ starting
 in the image of the local unstable manifold parameterization
 and ending at a point in the image of the local stable manifold.   
 The boundary conditions ensure that the orbit accumulates to the periodic orbit(s) 
 in forward and backward time thanks to the conjugacy relation \eqref{eq:flowConj}.
We observe however that solutions of the above system are not isolated, as 
if $\Gamma \colon [0, T] \to \mathbb{R}^6$ is one solution we obtain a
continuous family of other solutions $\Gamma_\tau \colon [0, T] \to \mathbb{R}^6$ 
by 
\[
\Gamma_\tau(t) = \Phi(\Gamma(t), \tau), 
\]
for any $|\tau| \ll 1$.

To isolate a solution we fix $\sigma_0, \phi_0$ to have length $R_1, R_2 \leq 1$
respectively.  This is 
equivalent to asking that the connecting orbit segment starts and finishes 
on a particular boundary torus of the local stable/unstable manifold, and 
this constraint removes the infinitesimal shift so that we have isolation.
To make this restriction explicit, we write 
\begin{align}\label{eq:preZero}
\begin{cases}
\dot \Gamma (t) = f(\Gamma(t)), & \forall t \in (0,T) \\
\Gamma(0)= P(\theta, R_1 \cos(\alpha),  R_1\sin(\alpha)), &\theta, \in [0, T_1],  \alpha \in [0, 2 \pi] \\
\Gamma(T)= Q(\phi, R_2 \cos(\beta), R_2 \sin(\beta)), & \phi \in [0, T_2], \beta \in [0, 2 \pi],
\end{cases}
\end{align}
where we remark that $R_1, R_2$ are not variables but fixed constants.
This is rewritten as a zero finding problem for $G \colon \mathbb{R}^5 \to \mathbb{R}^6$.
\begin{equation} \label{eq:zeroEq1}
G(T, \theta, \phi, \alpha, \beta) = \Phi(P(\theta, R_1 \cos(\alpha), R_1 \sin(\alpha)), T) - 
Q(\phi, R_2\cos(\beta), R_2\sin(\beta)), 
\end{equation}
where $\Phi$ is the flow generated by $f$.
While a zero of the system is isolated, we do not have a balanced system of 
equations hence cannot apply Newton's method. To balance the system we drop
any of the three components of the velocity. The choice depends on the trajectory of interest. Denote by $\hat{\Phi}$ and $\hat{Q}$ the flow and the local stable manifold parameterization
each with (for example) the sixth component omitted.  Then 
we define $\hat{G} \colon \mathbb{R}^5 \to \mathbb{R}^5$ by 
\begin{equation} \label{eq:zeroEq2}
\hat{G}(T, \theta, \phi, \alpha, \beta) = \hat{\Phi}(P(\theta, \cos(\alpha), \sin(\alpha)), T) - 
\hat{Q}(\phi, \cos(\beta), \sin(\beta)), 
\end{equation}
and note that Newton's method can be used to solve the problem.

Of course the flow $\Phi$ is only implicitly defined by the vector field $f$.  
We obtain an explicit zero finding problem as follows.  In anticipation of the 
discretization of the function spaces to follow we rescale time so that 
the orbit segment is on the image of $P$ at time $t = -1$ and on the 
image of $Q$ at time $t = 1$.
Define $\mathcal{F} \colon C([-1, 1], \mathbb{R}^5)
\to C([-1, 1], \mathbb{R}^5)$ by 
\begin{equation} \label{eq:connectingOrbit_functionalEquation}
\mathcal{F}(\Gamma, T, \theta, \phi, \alpha, \beta)(t) = 
\left(
\begin{array}{c}
\Gamma(t) 
- P(\theta, \cos(\alpha), \sin(\alpha)) - \frac{T}{2} \int_{-1}^t f(\Gamma(s)) \, ds \\
\hat{\Gamma}(1) - \hat{Q}(\phi, \cos(\beta), \sin(\beta))
\end{array}
\right).
\end{equation}
In practice we apply
Newton's method to $\mathcal{F}$ after discretizing $\Gamma$ using 
Chebyshev series as discussed in the 
next section.  Another technical detail is that since multiplication of 
Chebyshev series can be thought of 
as multiplication of cosine series, it is once again advantageous to work with the 
polynomial field discussed in the Appendix.

\subsection{Chebyshev discretization of the BVP}

After a translation and a rescaling of time, the solution of \eqref{eq:connectingOrbit_functionalEquation} 
is defined on $[-1,1]$ and therefore can be expressed using Chebyshev series for all nine 
component. As previously mentioned the use of Chebyshev expansion is well detailed in the 
literature and will lead to an operator defined on infinite sequences of coefficients 
which is similar to the definition from section \ref{sec:Manifold}. 

\begin{remark}\label{Remark:Chebyshev}
{\em
This approach, based on Chebyshev approximation, allow the use of a contraction mapping 
argument to validate the approximation. Such approach is already well known and had been
 the object of several studies. The interested reader can see for 
 example \cite{MR3392421,MR3207723,LessardReinhardt,MR3353132,paperBridge,RayJB}. }
\end{remark}

\begin{definition}\label{def:Chebyshev}
 Let $T_k:[-1,1] \to \mathbb{R}$ denotes the Chebyshev polynomials. 
 They satisfy the recurrence relation $T_0(t)=1$, $T_1(t)=t$ and
\[
 T_{k+1}(t)=2t T_k(t) - T_{k-1}(t),~ \forall k\geq 1.
\]
An analytic function $f:[-1,1] \to \mathbb{R}$ can be expressed uniquely as
\[
f(t) = a_0 + 2\sum_{k=1}^\infty a_k T_k(t),
\]
and it follows that the decay of the coefficients is exponential. 
Thus, the function $f$ is represented uniquely as an infinite sequence representing the coefficients of a Chebyshev expansion.
\end{definition}

It is possible to rewrite the solution as the zero of a well chosen infinite-dimensional operator defined on the space of Chebyshev coefficients
equivalent to the functional operator given by \eqref{eq:connectingOrbit_functionalEquation}. More detail about the rewriting of the 
problem can be found in the literature listed in Remark \ref{Remark:Chebyshev}. Let
\[
y= (L,\theta,\alpha,\phi, \beta, a^1,\hdots,a^9)
\]
where $L$ is the half-period ($L= \frac{T}{2}$), the pairs $\theta, \alpha$ and $\phi,\beta$ are coordinates for the unstable and stable parameterization 
of the boundary tori respectively, and $a^i $ are the coefficients of the Chebyshev expansion of each component of the solution. So that $y$ denotes the set of unknowns of the problem. We set
\begin{equation}\label{eq:OperatorBVP}
\mathcal{F}(y)= \left( \eta^1(y),\hdots,\eta^5(y),G^1(y),\hdots,G^9(y) \right),
\end{equation}
where each $\eta^i$ is a scalar equation arising from the rewriting 
the second line in \eqref{eq:connectingOrbit_functionalEquation}. 
We stress that each
$G^i$ is an infinite-dimensional equation to solve for the Chebyshev coefficients. 
The maps are explicitly defined as
\[
\eta^i(x)= \left(a_0^i +2\sum_{k=1}^\infty a_k^i \right) - P^i(\theta, R_1 \cos(\alpha), R_1 \sin(\alpha)),
\]
and
\[
G_k^i(y) =
\begin{cases}
 \left(a_0^i +2\sum_{j=1}^\infty (-1)^ja_j^i \right) - Q^i(\phi,  R_2 \cos(\beta), R_2 \sin(\beta)), & k=0, \\
 ka_k^i -L(F(a))_{k\pm1}^i, & k\geq 1,
\end{cases}
\]
where $(F(a))_{k\pm1}^i = (F(a))_{k+1}^i - (F(a))_{k-1}^i $. Each $F^i$ are similar to the case of the 
Fourier-Taylor parameterization of the manifold, and they are explicitly given by
\begin{align*}
F_{k}^{1}(a)=  &a_{k}^{2}, \\
F_{k}^{2}(a)=  & 2a_{k}^{4} +a_{k}^{1} -\sum_{i=1}^3 m_i\left((a^{1}-x_i)\star a^{6+i} \star a^{6+i}\star a^{6+i}\right)_{k},  \\ 
F_{k}^{3}(a)=  &a_{k}^{4}, \\
F_{k}^{4}(a)= & -2a_{k}^{2} +a_{k}^{3} -\sum_{i=1}^3 m_i\left((a^{3}-y_i)\star a^{6+i}\star a^{6+i}\star a^{6+i}\right)_{k}, \\
F_{k}^{5}(a)= &a_{k}^{6},  \\
F_{k}^{6}(a)= & -\sum_{i=1}^3 m_i\left((a^{5}-z_i)\star a^{6+i}\star a^{6+i} \star a^{6+i}\right)_{k}, \\
F_{k}^{7}(a)= & -\left( \left((a^{1}-x_1)\star a^2 +(a^3-y_1)\star a^4 +(a^5-z_1)\star a^6\right) \star a^7 \star a^7 \star a^7 \right)_{k}, \\
F_{k}^{8}(a)= & -\left( \left((a^{1}-x_2)\star a^2 +(a^3-y_2)\star a^4 +(a^5-z_2)\star a^6\right)\star a^8 \star a^8 \star a^8 \right)_{k}, \\
F_{k}^{9}(a)= &-\left( \left((a^{1}-x_3)\star a^2 +(a^3-y_3)\star a^4 +(a^5-z_3)\star a^6\right) \star a^9 \star a^9 \star a^9 \right)_{k},\\
\end{align*}
 where $\star$ denotes the convolution product. That is for $b= \left\{ b_k \right\}_{k=0}^\infty$ and $c= \left\{c_k \right\}_{k=0}^\infty$ two sequence of Chebyshev coefficients
 \[
 (b \star c)_{k} = \sum_{ \substack{ k_1+k_2 = k \\ k_1,k_2\in \mathbb{Z}}} b_{|k_1|}c_{|k_2|}.
 \]
Again, the coordinates of the primaries are written as Chebyshev series to simplify the presentation, the Chebyshev expansion of a constant being the constant itself as the first term and zeros for all the remaining coeffiients.

\begin{remark}[Rewriting of the problem]
Note that both sums in the definition of $\mathcal{F}$ arise from the evaluation of the trajectory $\Gamma$ at its endpoint as well as the fact that for all $k\geq 0$
\[
T_k(-1)= (-1)^k, \quad \mbox{and} \quad T_k(1)= 1.
\]
Moreover, the tridiagonal structure of each operator  $G^i$ arise from the fact that for all $k\geq 2$
\[
\int T_k(t)dt= \frac{1}{2} \left( \frac{T_{k+1}(t)}{k+1} -\frac{T_{k-1}(t)}{k-1} \right).
\]
To obtain the desired operator, one must use the integration formula for Chebyshev polynomials, simplify and then regroup matching coefficients.
For more details, we refer again to the literature in Remark \ref{Remark:Chebyshev}.
\end{remark}

\begin{remark}[domain subdivision]\label{rem:division}
For large values of $T$ the Chebyshev coefficients will decay slower and the finite dimensional approximation can loose accuracy. While one can use a higher dimensional approximation it is often more efficient to divide the domain. We will exhibit how one can split the domain in half, this process can be repeated to divide the domain into as many pieces as desired. Recall that the original problem is to find $\Gamma: [0,T] \to \mathbb{R}^6$ satisfying \eqref{eq:preZero}. This problem is completely equivalent to the following two boundary problems. Let $0 < \tilde T < T$ and seek a pair of function $\Gamma_1:[0,\tilde T] \to \mathbb{R}^6$, $\Gamma_2:[\tilde T, T] \to \mathbb{R}^6$ satisfying
\begin{align*}
\begin{cases}
\dot \Gamma_1 (t) = f(\Gamma_1(t)), & \forall t \in (0,\tilde T) \\
\Gamma_1(0)= P(\theta, R_1 \cos(\alpha),  R_1\sin(\alpha)), &\theta, \in [0, T_1],  \alpha \in [0, 2 \pi] \\
\Gamma_1(\tilde T)= \Gamma_2(\tilde T),
\end{cases}
\end{align*}
and
\begin{align*}
\begin{cases}
\dot \Gamma_2 (t) = f(\Gamma_2(t)), & \forall t \in (\tilde T, T) \\
\Gamma_2(\tilde T)= \Gamma_1(\tilde T) , \\
\Gamma_2(T)= Q(\phi, R_2 \cos(\beta), R_2 \sin(\beta)), & \phi \in [0, T_2], \beta \in [0, 2 \pi],
\end{cases}
\end{align*}
Note that $\Gamma_1$,$\Gamma_2$ are restriction of the original trajectory to smaller time. In order to construct a two point boundary value problem for each piece we use the fact that $\Gamma$ is continuous so that the restrictions must match at the transition point. A natural choice of transition point is to set $\tilde T = \frac{T}{2}$, however the accuracy of the solution can sometime be improved using a nonuniform mesh.

Both subdomains are then transformed into $[-1,1]$ and expanded using  Chebyshev series.
 Using two Chebyshev expansions would double the total number of variables in the problem, 
 although the gain in the decay rate of each sequence often allows to reduce the projection of the individual Chebyshev expansion,
 resulting in the use of fewer total modes.
\end{remark}

To find a connection, we apply Newton's method to a finite dimensional projection of the problem. 
Given a pair of manifold one can compute an approximate zero of the operator and then use 
definition \ref{def:Chebyshev} to display the approximate connection. 
Below is a sketch of the procedure.

\begin{enumerate}
	\item Pick two periodic orbit $\gamma_1(t)$, $\gamma_2(t)$ and verify that 
	$J(\gamma_1(t))=J(\gamma_2(t))$ if the orbits are distinct.
	\item Verify that both periodic orbit have the desired stability. Fix a scale for the 
	tangent bundles as well as the desired dimension for the finite dimension approximation 
	in both the Fourier and Taylor direction. Note that greater values for the scale of the bundle 
	will require a higher choice for the Taylor direction to maintain sufficient accuracy, but it will 
	also reduces the integration time required to find connecting orbits.
	\item Compute $P(\theta, \sigma)$ the parameterization of the local unstable manifold attached 
	to $\gamma_1(t)$ and $Q(\phi,\rho)$ the parameterization of the local stable manifold 
	attached to $\gamma_2(t)$.
	\item Define the following positive constants $d_{\mbox{max}}$, $\Delta t$, $T_{\mbox{max}}$ 
	and construct a triangulation of the boundary of both manifold such that the average length of 
	the edges of every triangle is less than $d_{\mbox{max}}$. Note that the boundary of the 
	manifold is given by the case $R_1= R_2=1$. Denote by 
	$\mathcal{T}_0^u = \left\{ P(\theta_i,\cos(\alpha_i), \sin(\alpha_i)) : i \in \mathcal{I} \right\}$ the set of vertex 
	of the triangulation. Similarly $\mathcal{T}_0^s$ will denote the case associated to the stable manifold.
	\item For every $p \in \mathcal{T}_0^u$, use numerical integration to obtain $\Phi(p,\Delta t)$ and use
	 the resulting point to define $\mathcal{T}_{t_1}^u$, where $t_1 = t_0 +\Delta t = 0 +\Delta t$. Refine the mesh by subdividing triangle with 
	 average edge length greater than $d_{\mbox{max}}$. To subdivide an edge, note that at this step
	  the two vertices are given by $\Phi(P(\theta_1, \cos(\alpha_1), \sin(\alpha_1)),t_1)$ and 
	  $\Phi(P(\theta_2,\cos(\alpha_2), \sin(\alpha_2)),t_1)$, we approximate the midpoint of
	   the edge by taking the image of the 
	  midpoint in parameter space. That is, we take $\Phi(P(\theta_3,\cos(\alpha_3), \sin(\alpha_3)),t_1)$ with 
	  $\theta_3= \frac{ \theta_1 + \theta_2}{2}$ and $\alpha_3 = \frac{\alpha_1 + \alpha_2}{2}$.
	   Similarly, we compute and refine the set $\mathcal{T}_{t_1}^s$, but in this case by integrating 
	   numerically backwards in time.
	\item Find the pair minimizing the distance between the two set. That is $(\theta,\alpha)$ 
	and $(\phi,\beta)$ such that 
	\[
	\left\|  \Phi(P((\theta,R_1 \cos(\alpha), \sin(\alpha),t_1) -\Phi(Q((\phi,R_2 \cos(\beta), \sin(\beta)),-t_1)  \right\|
	\]
	is minimal. If the minimum distance is sufficiently small, then set $L=t_1$ and the algorithm 
	provides an initial guess for the use of Newton's method to obtain an approximate zero of
	 the operator $\mathcal{F}$ given in \eqref{eq:OperatorBVP}.
	\item While $t_n<T_{\mbox{max}}$, repeat Step $5$ to define $\mathcal{T}_{t_{n+1}}^u$ and 
	$\mathcal{T}_{t_{n+1}}^s$. Then, repeat Step $6$ to obtain the candidate and test the existence 
	of a nearby approximate zero using Newton's method.
\end{enumerate}

\begin{remark}[The case of collision]
{
It is possible for some points in the triangulation to reach a collision, such occurrence make the size 
of the sets $\mathcal{T}_{t_{n}}^u$ and $\mathcal{T}_{t_{n}}^s$ grow considerably. In the present work  
we reject such occurrences by adding the following constraint to
 step $5$ of the algorithm. Let $v_{\mbox{max}}$ and $d_{\mbox{lib}}$ be positive constants and let 
 $\mathcal{L}_i$ denote the libration point shadowed by the periodic orbit $\gamma_1(t)$. Reject all 
 points $p$ of $\mathcal{T}_{t_{n}}^u$ and $\mathcal{T}_{t_{n}}^s$ such that 
 $\sqrt{ p_2^2 +p_4^2 +p_6^2} > v_{\mbox{max}}$ or $\left\| p - \mathcal{L}_i \right\|> d_{\mbox{lib}}$. 
 The first condition rejects collisions since one can easily notice that any trajectory approaching one 
 of the primaries will have large velocity. The second condition rejects trajectories escaping a chosen 
 neighborhood of the libration points, with such a criteria we note that the algorithm cannot find 
 connecting orbit with really large flying time. The choice of $v_{\mbox{max}}$ and $d_{\mbox{lib}}$
  is guided by a priori simulation of the system and the intent to speed up the algorithm as much 
  as possible.}
\end{remark}

\begin{remark}
{
The algorithm is useful to determine the length of the shortest existing connection in a specific 
case but its accuracy is highly dependent on the values of the constants $d_{\mbox{max}}$, 
$\Delta t$ as well as the accuracy of the ODE solver used to numerically integrate the problem. 
A more thorough study could be provided by a generalization of the approach used in the planar 
case. That work is explained with more details in \cite{MR3919451}. The generalization requires 
to change the basis for the periodic direction from Fourier to Chebyshev approximation, this choice 
of basis is the object of \cite{chebManifolds}. This extension is the subject of 
work in preparation by the authors.}
\end{remark}

\subsection{Numerical Example}

An example of the results obtained in the spatial CRFBP using this procedure is illustrated 
in Figure \ref{fig:BVP}.  Here we computed the local stable/unstable manifold parameterizations 
to polynomial order $5$ taking $20$ Fourier coefficients to represent each of the Taylor 
coefficients (including the periodic orbit and normal bundles). The trajectory $\Gamma(t)$ is 
represented using two distinct Chebyshev expansion, each expansion has $50$ coefficients 
for a total number of $905$ unknowns to use Newton's method. The total time  of flight of the connecting orbit 
is
\[
T= 3.4698
\]
and both Chebyshev problem have equal time, that is $\tilde T = \frac{T}{2}$ using the technique from 
Remark \ref{rem:division}. Let $\bar{\mathcal{F}}$ denote the finite dimensional projection of $\mathcal{F}$, 
with dimension $905$ in this case, and $\bar y$ denote the numerical approximation. Newton's method provided 
an approximation with defect close to machine precision, that is
\[
\bar{\mathcal{F}}(\bar y) \approx 10^{-15}.
\]
For the computation the fourth component was dropped and after a posteriori verification it was validated that both 
component are equal. The flying time depends on the scale of the manifolds, in that specific case the scale chosen
 is $0.1$, so that the tangent bundle at initial time have length of approximatively $0.1$. This condition is not applied 
 exactly and instead approximated in Fourier coefficients by
\[
\sum_{i=1}^9 \sum_{|k|< k_0}\left( a_{\alpha,k}^{(i)}\right)^2 \approx 1,
\]
where $\alpha= (1,0)$ or $(0,1)$. This condition was applied with $k_0=5$. We remark that this specific connecting orbit
 reached a maximum velocity of approximatively $1.81$ and a  
maximum distance from $\mathcal{L}_5$ of $1.1$. 

The connecting orbits in the remainder of the paper are computed using the procedure just 
discussed and the numerical details are similar.  

\begin{figure}
    \subfigure{{\includegraphics[width=.45\textwidth]{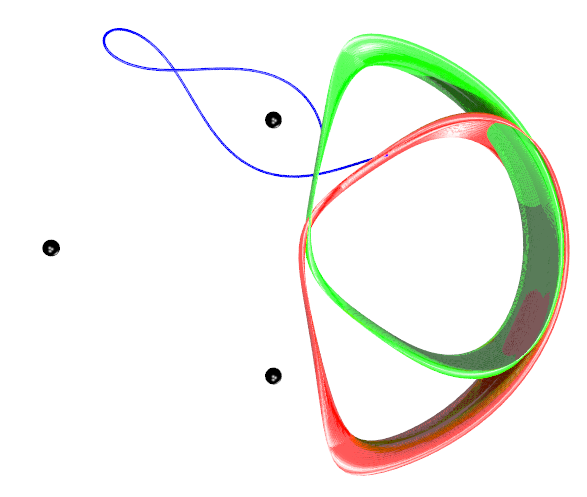}}}
     \subfigure{{\includegraphics[width=.45\textwidth]{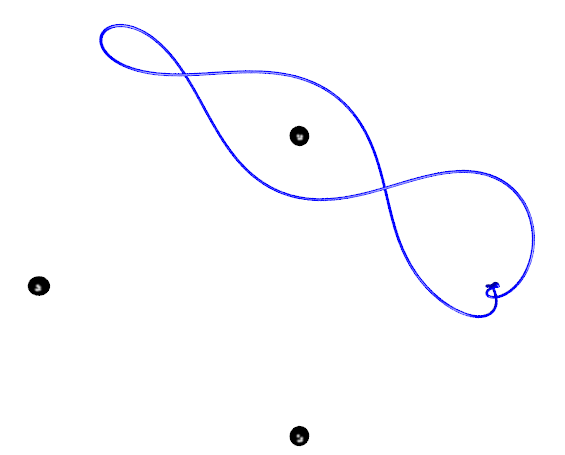} }}
\caption{\textbf{Example -- BVP for a vertical Lyapunov homoclinic in practice:}
(Left) Representation of the boundary value problem for the case of a 
homoclinic connection to a periodic orbit at $\Lc_5$. The boundary torus of the stable 
manifold is represented in green, with the boundary torus of the unstable manifold in red.
The Chebyshev arc in blue. Both surface are displayed using the same map as in 
Figure \ref{fig:ManifoldExample}, 
this time with $R_1=1$ for the unstable case and $R_2=1$ for the stable. 
We remark that the apparent intersection of the local parameterizations 
in the right side of the left frame is due to projection distortions which arise when projecting 
from the six dimensional phase space to the three dimensional configuration space.
(Right) The full 
connecting orbit is recovered using the flow conjugacy relation on the 
local parameterizations, that is the asymptotic behavior is obtained without
integrating the CRFBP.}\label{fig:BVP}
\end{figure}

\section{Results: Homoclinic connections for the vertical Lyapunov families in the CRFBP} \label{sec:results}
We now return to the main goal of the present work, and apply the numerical algorithms developed
in the previous sections to the homoclinic connection problem at $\mathcal{L}_{0,5}$ in 
the CRFBP for mass ratios at or near the triple Copenhagen problem.

See for example the results illustrated in Figure \ref{fig:L0spacial315}.  Here we have taken the 
masses of the primaries to be $m_1 = 0.4, m_2 = 0.35$, and $m_3 = 0.25$, so that the 
$\pm 120$ degree symmetry is broken.  We consider the vertical Lyapunov family at 
$\mathcal{L}_0$ which lies near, but not on the $z$-axis thanks to the broken symmetry. 
As expected we find that the three shortest homoclinic orbits have the shape
predicted by the planar problem.  See for example the orbit in the left frame of 
Figure \ref{fig:homoclinics_L0} in the present work.  See also the top right frame 
of Figure 21 in \cite{MR3919451}, which illustrates the shortest 
planar homoclinics at almost the same parameter values as used here. Note however 
that when we view the orbits from the side in the $x,y,z$ spatial coordinate frame 
we see that the orbits have substantial out of plane amplitude (5-10 percent of the 
$xy$ amplitudes), despite the fact that the 
$xy$ projection fits well with the planar case.

\begin{figure}[!t]
     \subfigure{{\includegraphics[width=.5\textwidth]{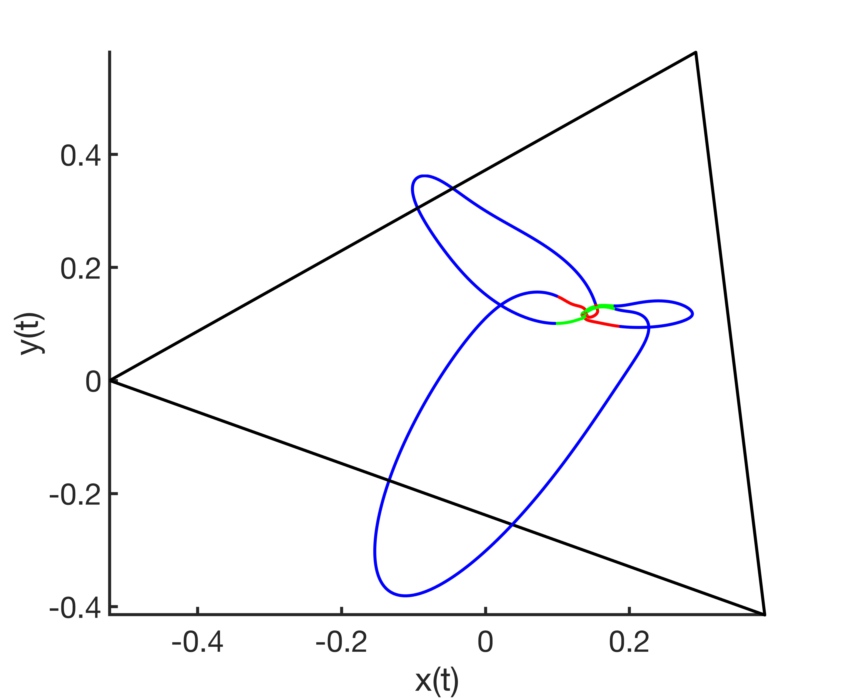} }}
      \subfigure{{\includegraphics[width=.5\textwidth]{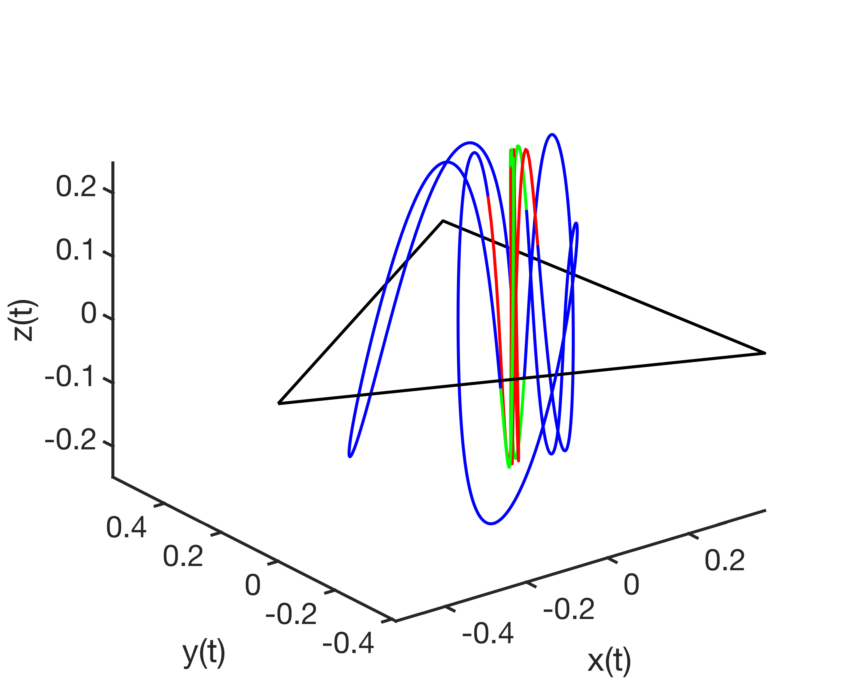}}}
\caption{\textbf{Homoclinic connections to a spatial periodic orbit from the 
vertical Lyapunov family at $\Lc_0$:}  Mass values of $m_1=0.4$ and $m_2=0.35$. 
$J= 3.15$ is the Jacobi constant of the periodic orbit.  Compare the shapes of the orbits in the left 
frame with the planar homoclinic orbits in the  top right frame of Figure 21 in \cite{MR3919451},
or (more loosely) with the shapes of the planar homoclinics in the left frame
of Figure \ref{fig:homoclinics_L0} of the present work.  While the shape of the cycle-to-cycle 
homoclinics are clearly inherited from the shapes of the planar homoclinics,  
the right frame illustrates the out of plane dynamics of the spatial homoclinic
tangle.  
}\label{fig:L0spacial315}
\end{figure}

\begin{figure}[!t]
     \subfigure{{\includegraphics[width=.5\textwidth]{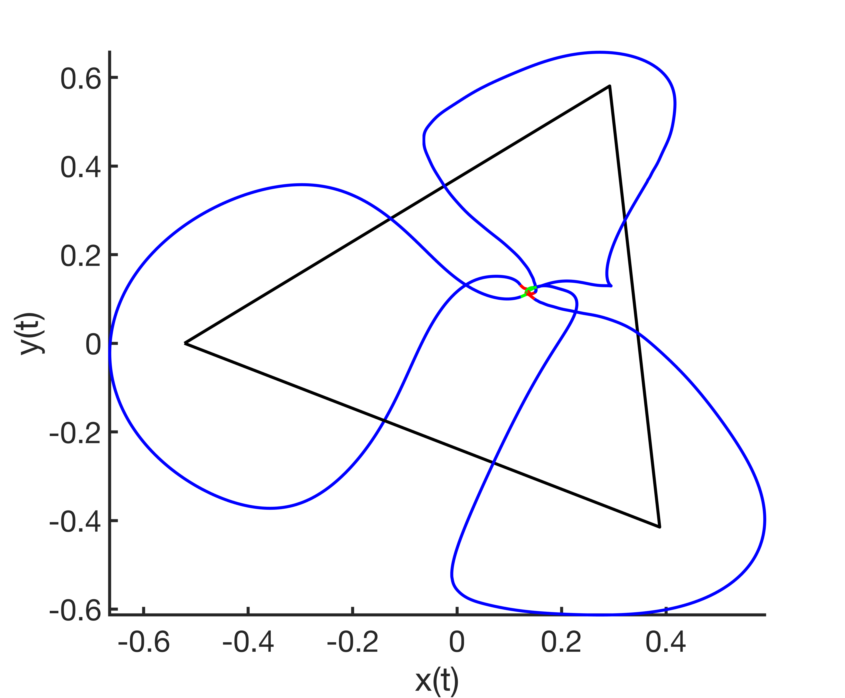} }}
      \subfigure{{\includegraphics[width=.5\textwidth]{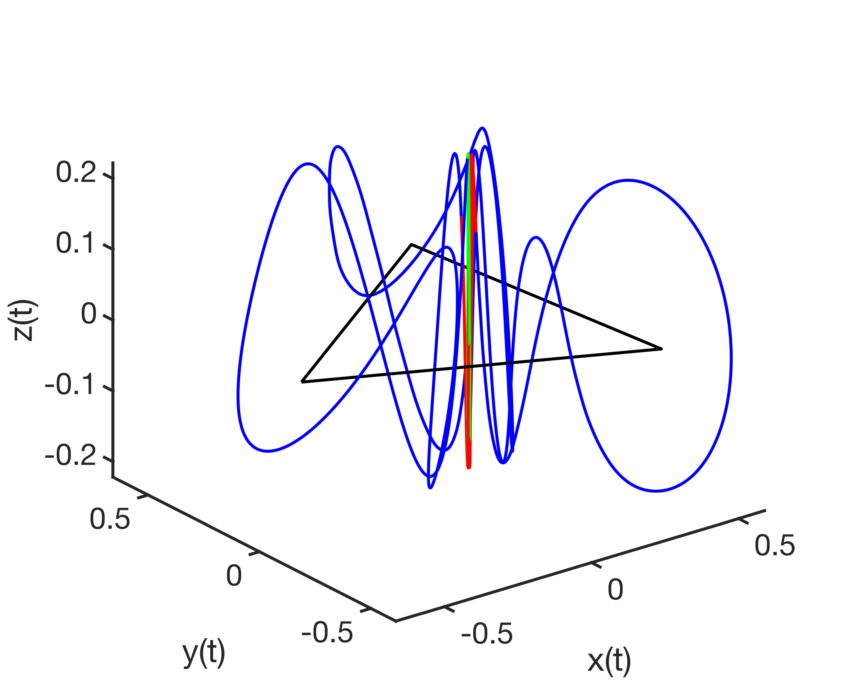}}}
\caption{Homoclinic connections to a spatial periodic orbit from the 
vertical Lyapunov family at $\Lc_0$.  Mass values of $m_1=0.4$ and $m_2=0.35$. 
$J= 3.2$ is the Jacobi constant of the periodic orbit.  Compare the shapes of the orbits in the left 
frame with the planar homoclinic orbits in the bottom right frame of Figure 21 in \cite{MR3919451},
or (more loosely) with the shapes of the planar homoclinics in the right frame
of Figure \ref{fig:homoclinics_L0} of the present work.
While the shape of the cycle-to-cycle homoclinics are clearly inherited from
the planar homoclinics,
the right frame illustrates the out of plane dynamics of the spatial homoclinic
tangle.  
 }\label{fig:L0spacial}
\end{figure}

The story is much the same for the fourth, fifth, and sixth shortest connections as
illustrated in  Figure \ref{fig:L0spacial}.  Again the homoclinic orbits have the 
shape predicted in the planar problem, as seen by considering the 
right frame of Figure \ref{fig:homoclinics_L0} in the present work, and also the 
bottom right frame of Figure 21 in \cite{MR3919451}.  At the same time it is 
important to remark once again that when viewed in the spatial problem we 
see that the cycle-to-cycle connections have substantial out of plane amplitude.

The situation is similar at $\mathcal{L}_5$.  Figures 
\ref{fig:homoclinicL5B},
\ref{fig:homoclinicL5A}, and \ref{fig:homoclinicL5}
illustrate the situation in the triple Copenhagen problem
with $m_1 = m_2 = m_3 = 1/3$.  Since the $\pm 120$ degree
rotational symmetry is not broken the dynamics are the same 
at $\mathcal{L}_{4,6}$.   The figures should be compared 
with Figure \ref{fig:homoclinics_L6} of the present work, which 
illustrates that indeed the shapes of the spatial cycle-to-cycle homoclinics
are in strong agreement with the planar saddle focus equilibrium homoclinics. 
The right frame of each of Figures 
\ref{fig:homoclinicL5B},
\ref{fig:homoclinicL5A}, and \ref{fig:homoclinicL5}
illustrates the out of plane motion of each homoclinic and the convergence
to the vertical Lyapunov orbit.

%
%

%

\begin{figure}[!t]
 \subfigure{{\includegraphics[width=.45\textwidth]{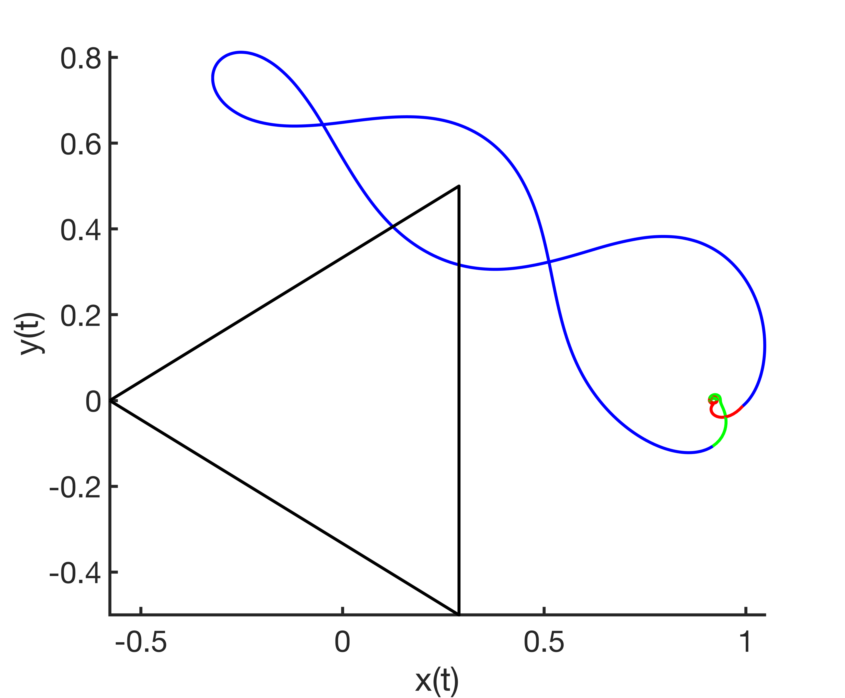} }} 
    \subfigure{{\includegraphics[width=.45\textwidth]{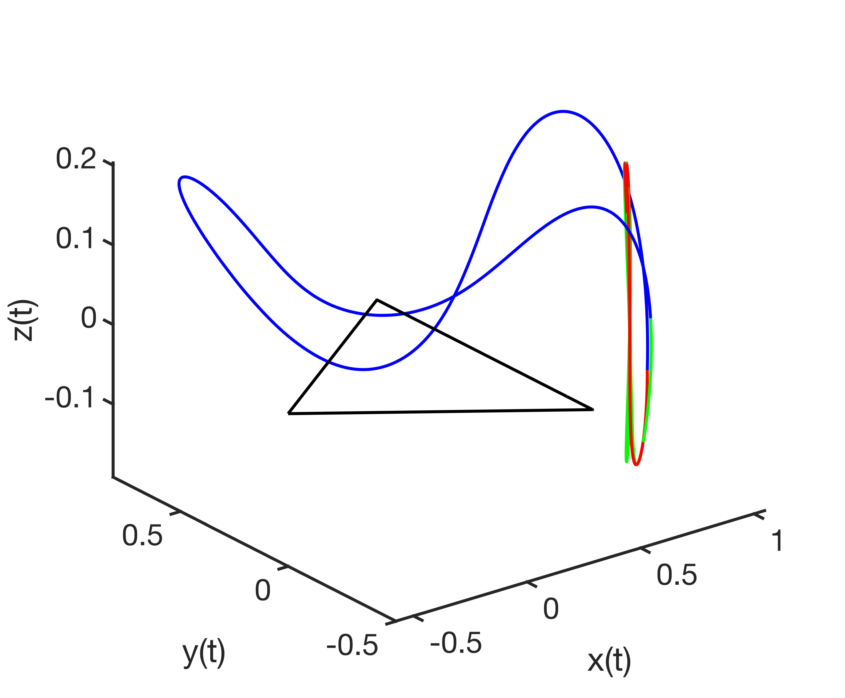}}}\\
     \subfigure{{\includegraphics[width=.45\textwidth]{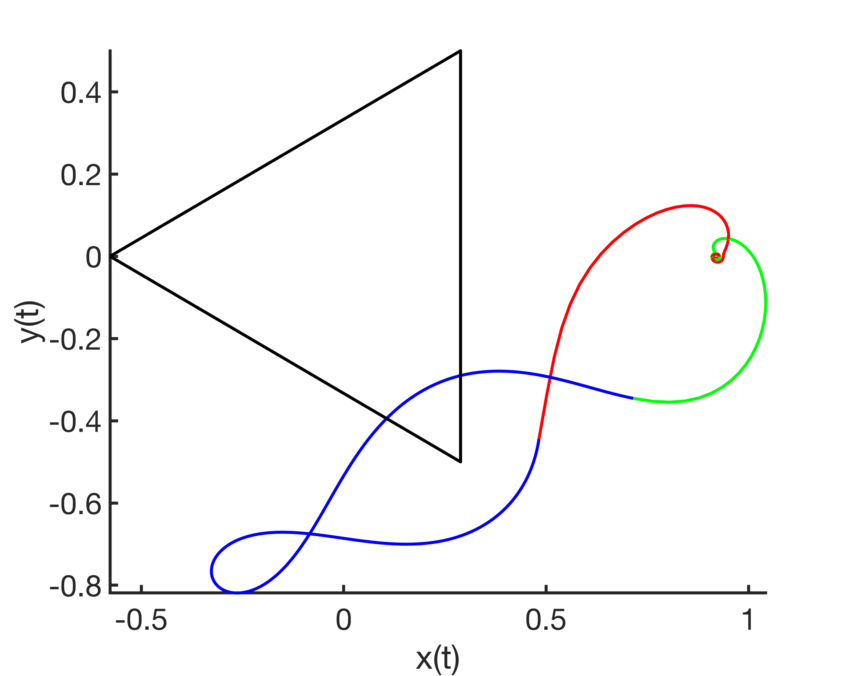} }} 
     \subfigure{{\includegraphics[width=.45\textwidth]{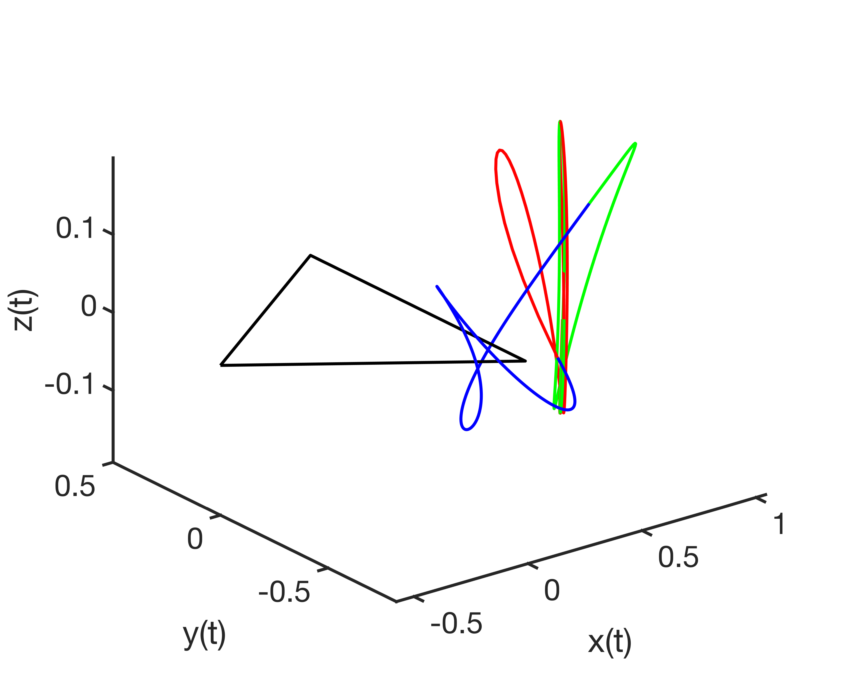}}}
       \caption{
       Homoclinic connections to a spatial periodic orbit from the 
vertical Lyapunov family at $\Lc_5$ in the triple Copenhagen problem.
Compare the shapes of the orbits in the left 
frames with the planar homoclinic orbits in the top two frames of 
Figure \ref{fig:homoclinics_L6} of the present work 
(Figure 18 in \cite{MR3919451}).
While the shape of the cycle-to-cycle homoclinics are clearly inherited from
the planar homoclinics,
the right frame illustrates the out of plane dynamics of the spatial homoclinic
tangle.       
}\label{fig:homoclinicL5B}
\end{figure}

\begin{figure}[!t]
\subfigure{{\includegraphics[width=.45\textwidth]{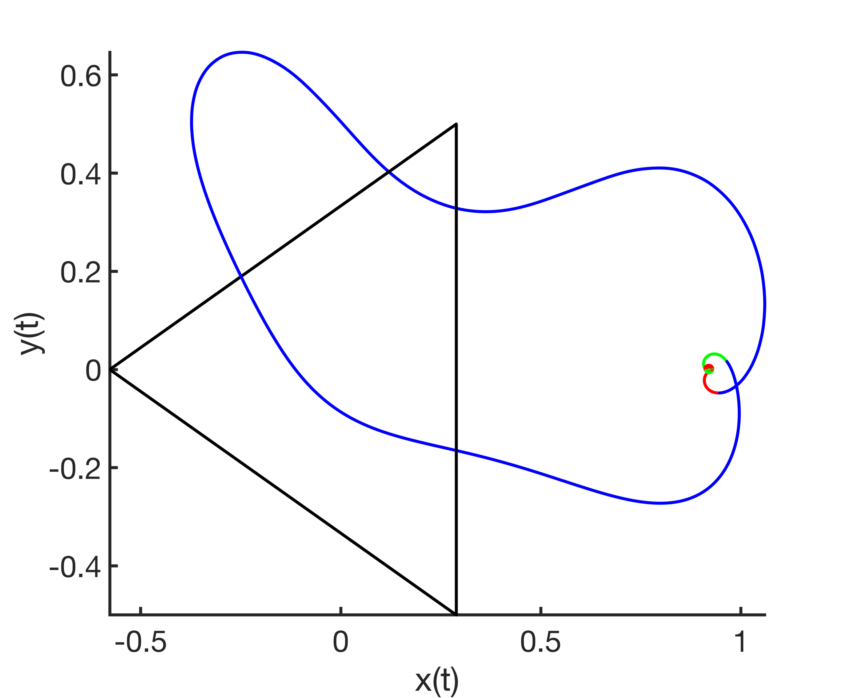} }}
     \subfigure{{\includegraphics[width=.45\textwidth]{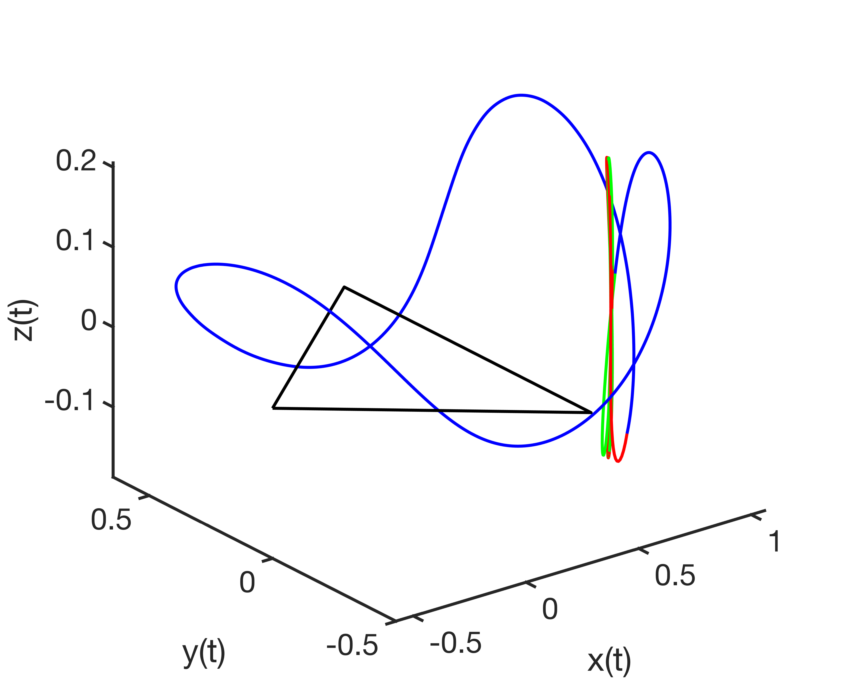}}} \\
       \subfigure{{\includegraphics[width=.45\textwidth]{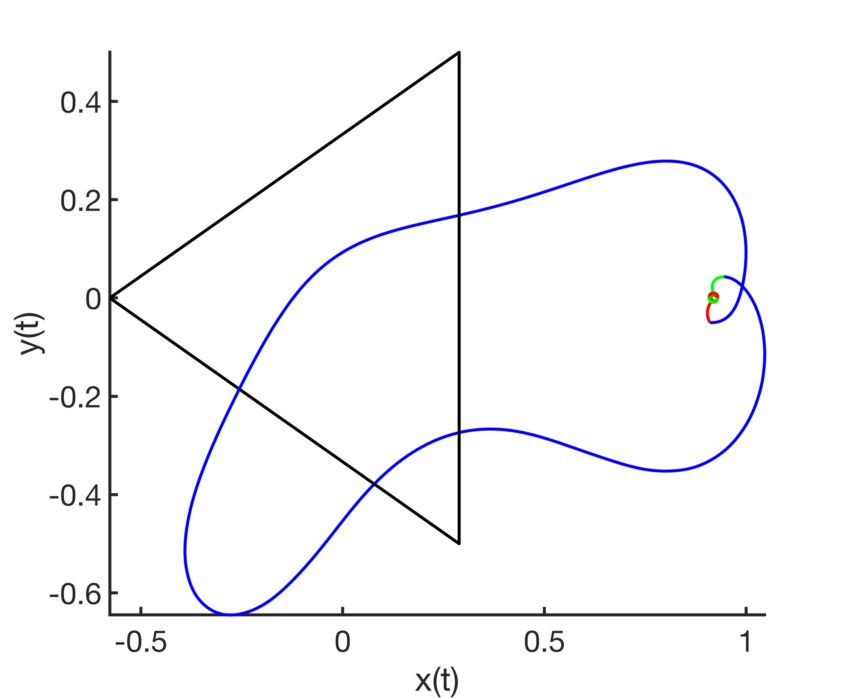} }} 
     \subfigure{{\includegraphics[width=.45\textwidth]{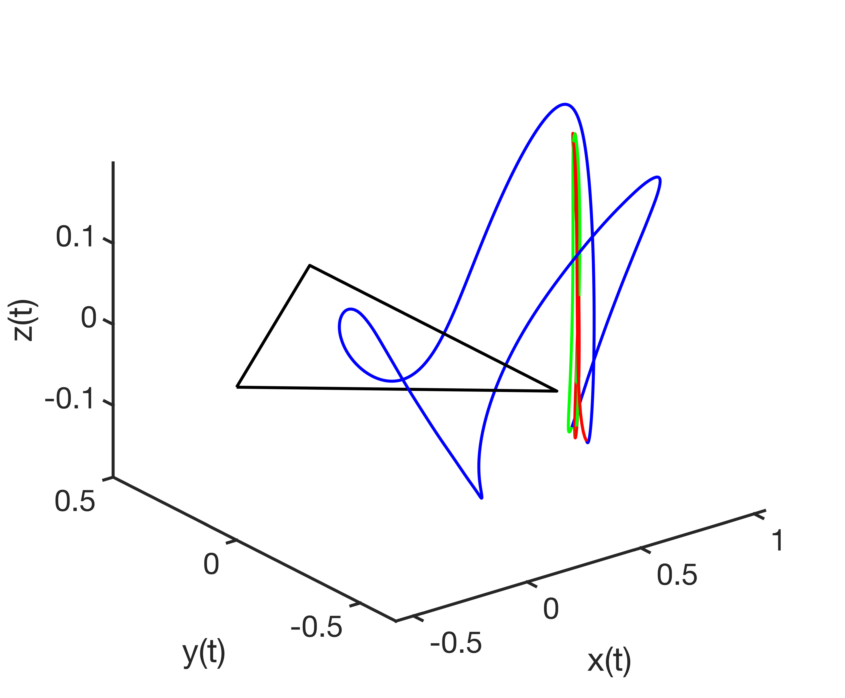}}}
     \caption{
          Homoclinic connections to a spatial periodic orbit from the 
vertical Lyapunov family at $\Lc_5$ in the triple Copenhagen problem.
Compare the shapes of the orbits in the left 
frames with the planar homoclinic orbits in the middle two frames of 
Figure \ref{fig:homoclinics_L6} of the present work 
(Figure 18
in \cite{MR3919451}).
While the shape of the cycle-to-cycle homoclinics are clearly inherited from
the planar homoclinics,
the right frame illustrates the out of plane dynamics of the spatial homoclinic
tangle.       
      }\label{fig:homoclinicL5A}
\end{figure}

\begin{figure}[!t]
  \subfigure{{\includegraphics[width=.5\textwidth]{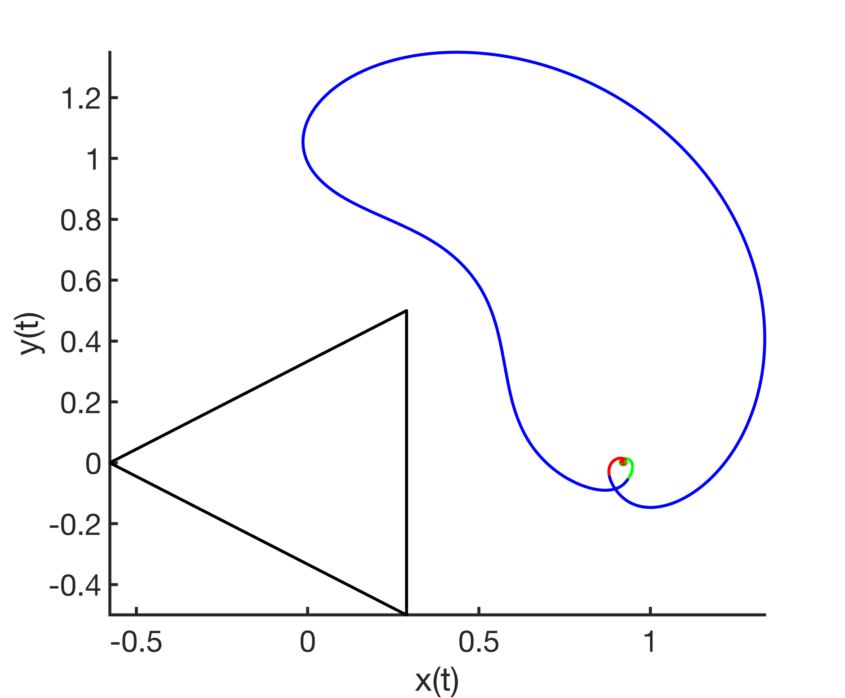} }}
     \subfigure{{\includegraphics[width=.5\textwidth]{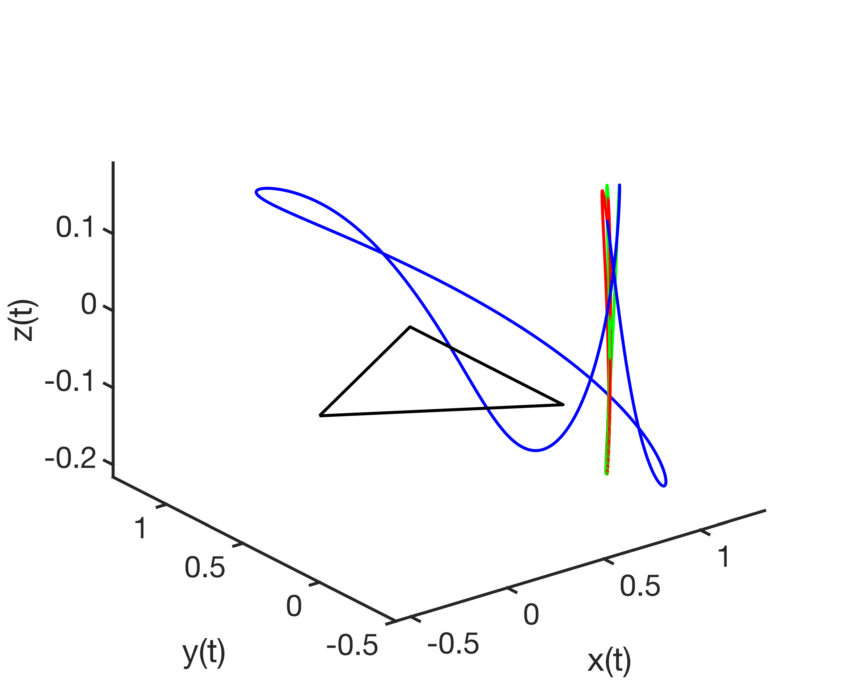}}} \\
       \subfigure{{\includegraphics[width=.5\textwidth]{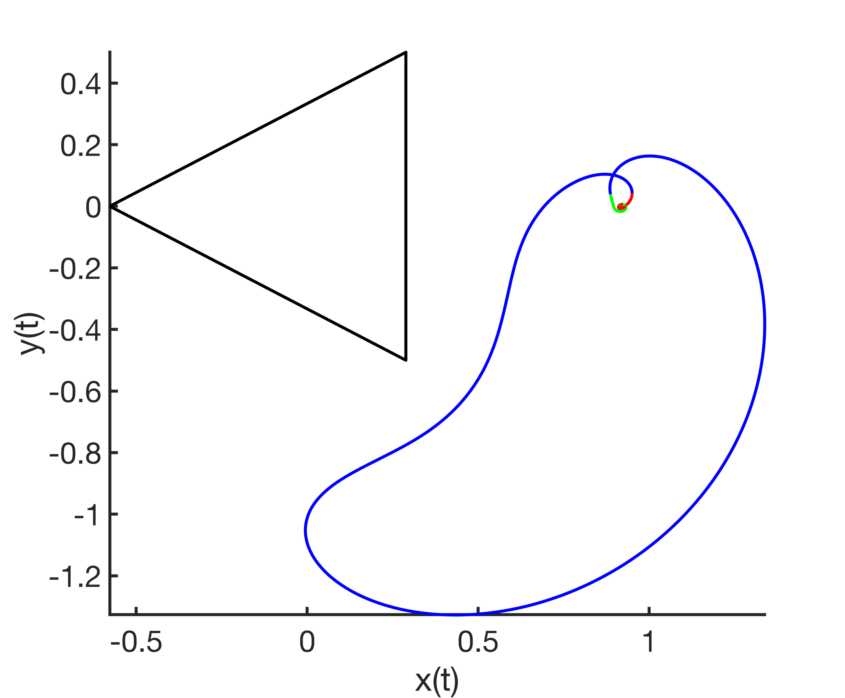} }}
     \subfigure{{\includegraphics[width=.5\textwidth]{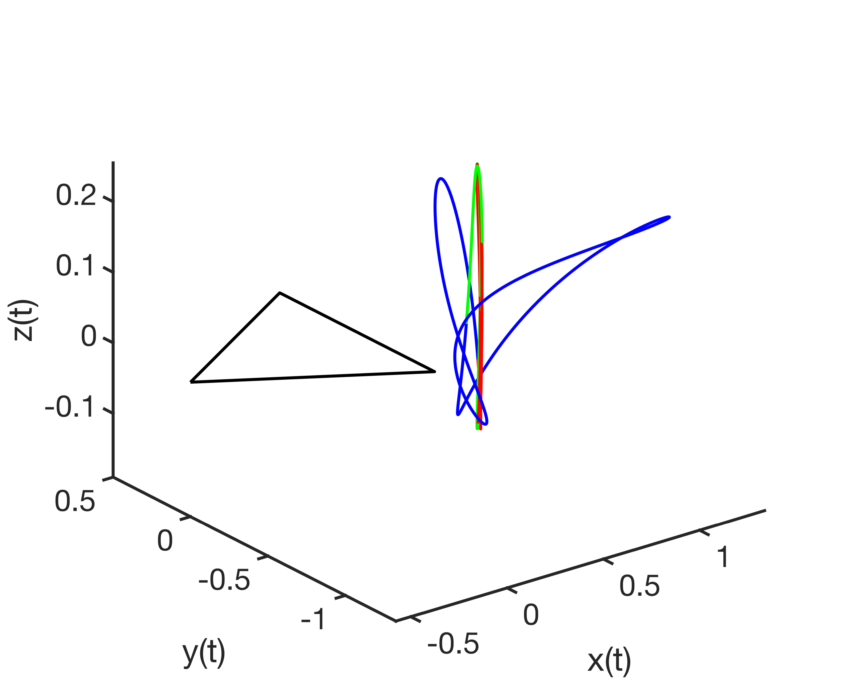}}} 
\caption{
     Homoclinic connections to a spatial periodic orbit from the 
vertical Lyapunov family at $\Lc_5$ in the triple Copenhagen problem.
Compare the shapes of the orbits in the left 
frames with the planar homoclinic orbits in the bottom two frames of 
Figure \ref{fig:homoclinics_L6} of the present work 
(Figure 18 in \cite{MR3919451}).
While the shape of the cycle-to-cycle homoclinics are clearly inherited from
the planar homoclinics,
the right frame illustrates the out of plane dynamics of the spatial homoclinic
tangle.        
}\label{fig:homoclinicL5}
\end{figure}

We illustrate in Figure \ref{fig:homoclinicL0continued} that the spatial dynamics just discussed
hold for nearby values of the Jacobi integral.  That is, the shapes are robust for nearby 
energies.  Continuation in the masses leads to similar robustness results.

\begin{figure}[!t]
    \subfigure{{\includegraphics[width=.5\textwidth]{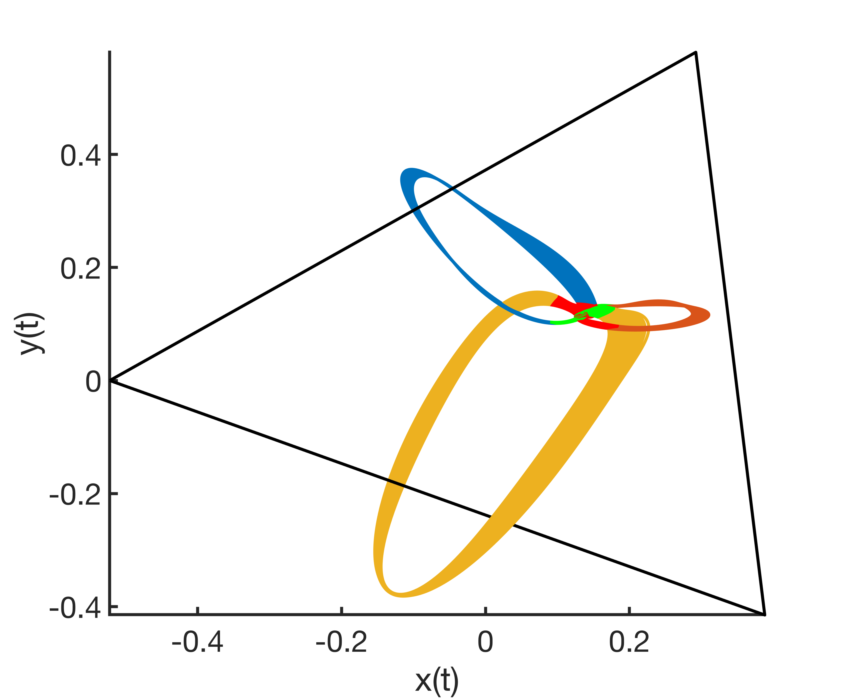}}}
    \subfigure{{\includegraphics[width=.5\textwidth]{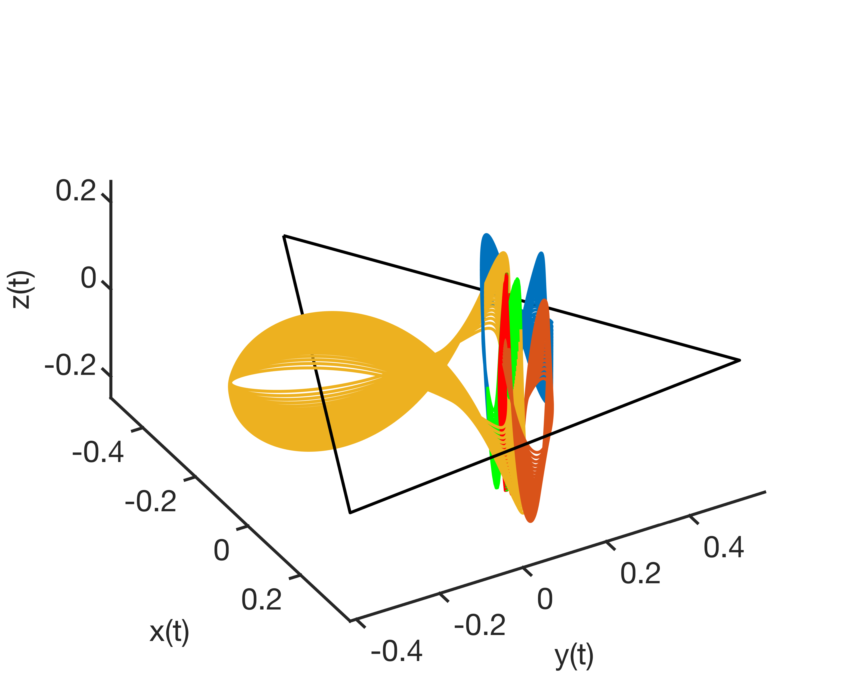}}}
\caption{\textbf{Numerical continuation of homoclinic connections for the spatial Lyapunov family 
   at $\mathcal{L}_0$:} Masses $m_1=0.4$, $m_2=0.35$ and values of the energy from 
     $J=3.15$ to $J=3.4$. The part displayed in green and red are given using the conjugacy relation and the Parameterization of the stable and unstable manifold respectively. The parameterization are approximated up to order $8$ in Taylor and order $20$ in Fourier. The part displayed in blue, orange and yellow are all given by the Chebyshev expansion. The three families of connecting orbit are out of plane and seem to accumulate to the planar homoclinic orbit displayed in section \ref{sec:Planar}. 
     The calculation suggests that the cycle-to-cycle chaos continuing out of the 
     planar homoclinic web persists for moderately large out of plane amplitudes.
      }\label{fig:homoclinicL0continued}
\end{figure}

Finally we provide some numerical indication that the picture does change dramatically 
when the Jacobi integral is changed enough.  For example the results in Figure 
\ref{fig:L0spacial255} shows the four shortest connecting orbits at $\mathcal{L}_0$ 
when $m_1 = 0.4, m_2 = 0.35$, and $m_3 = 0.25$, for $J = 2.55$.  We note that the 
height of the vertical Lyapunov periodic orbit at this value of the energy has more than 
doubles compared with the results in Figures \ref{fig:L0spacial315} and \ref{fig:L0spacial},
and that the shortest connections are dramatically shorter.  

Figure 
\ref{fig:L0spacial255} shows the same four orbits in the $xy$ projection, 
and we see that the shapes of the spatial cycle-to-cycle homoclinics are no longer
described by the planar problem.  We know that the $\mathcal{L}_2$ vertical family 
is very close to the $\mathcal{L}_0$ family at this value of energy. We conjecture that
there are heteroclinic cycles between the two familes at this energy, and that the 
short homoclinics shadow these connections.

\begin{figure}[!t]
    \includegraphics[width=\textwidth]{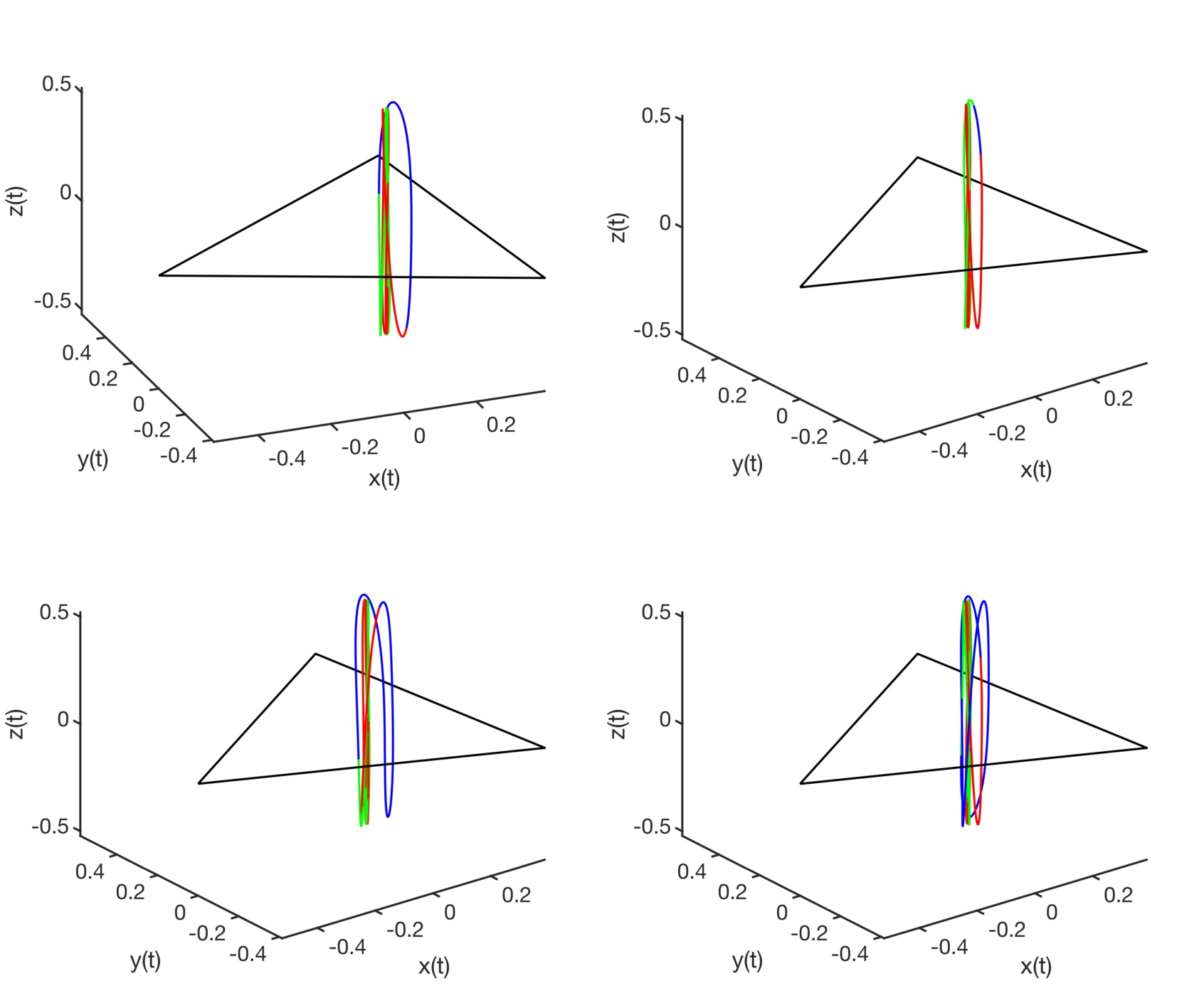}
\caption{\textbf{Four shortest homoclinic connections to a periodic orbit at $\Lc_0$-- large out of 
plane amplitude:} 
Mass values of $m_1=0.4$ and $m_2=0.35$ and $J= 2.55$ -- much higher out of plane 
amplitude than considered for the results illustrated in Figures
\ref{fig:L0spacial315} and \ref{fig:L0spacial}.  The results suggest a dramatic change in the 
phase space structure, as the shortest homoclinics no longer resemble the planar 
case.  The likely explanation is that the vertical Lyapunov family at $\mathcal{L}_2$ 
is close to the vertical family at $\mathcal{L}_0$, and that the stability of both are
saddle focus.  There are likely heteroclinic connections between these two families,
and the homoclinics shown here shadow these heteroclinics.  
}\label{fig:L0spacial255}
\end{figure}

\begin{figure}[!t]
    \includegraphics[width=0.7\textwidth]{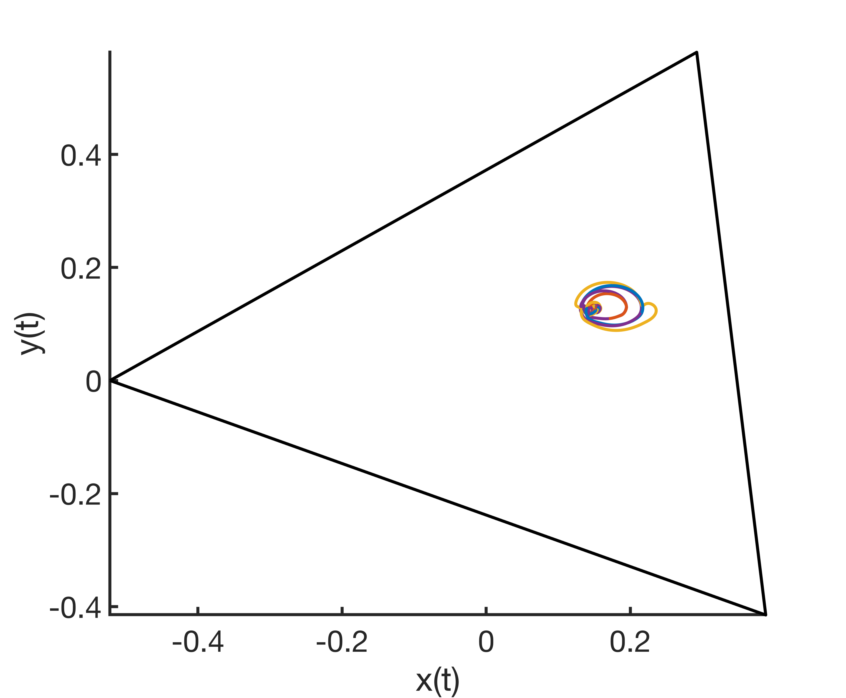}
\caption{View from above of all four shortest homoclinic connections 
displayed in Figure \ref{fig:L0spacial255}.  }
\end{figure}

\section{Conclusions} \label{sec:conclusions}
The numerical results given in the main body of the present work 
provide an example of a case where the existence of a blue sky 
catastrophe in the planar subsystem gives rise to chaotic motions
in the full spatial problem.  This is a very interesting phenomena 
because while the blue sky catastrophe can appear only at 
discrete values of the Jacobi integral of the planar system --
the energies of the saddle focus equilibrium solutions -- 
the transverse cycle-to-cycle homoclinics are robust with respect
to small perturbations in the energy.  The dynamics in the planar
system at the saddle focus energy level have  
ramifications for the dynamics of the spatial problem over a whole 
range of  energies away from the planar value.

The discussion can be formalized as follows.  Suppose that a three degree 
of freedom Hamiltonian system has (A) an invariant planar (two degree of freedom)
subsystem, (B) an in plane equilibrium solution whose linear stability 
is saddle-focus relative to the invariant plane and saddle-focus $\times$ center
in the full problem, (C) an in plane orbit homoclinic to the saddle-focus
equilibrium.  Then:
\begin{itemize}
\item \textbf{$1$(P)} There exists an invariant tube of planar periodic orbits, parameterized by energy, accumulating to the homoclinic orbit.  The stability of the orbits change infinitely many times as they approach the homoclinic along the tube. 
\item \textbf{$2$(P)} There are chaotic dynamics in a neighborhood of the homoclinic. The chaotic subsystem is an invariant subset of the plane in the energy level of the equilibrium.  
\item \textbf{$3$(C)}  There is a one parameter family of out of plane periodic 
orbits in the center manifold of the equilibrium.  In a small enough 
neighborhood of the equilibrium these periodic orbits
orbits have saddle-focus stability and hence three dimensional stable/unstable manifolds.
For any out of plane periodic orbit with small enough out of plane amplitude,
the stable/unstable manifolds of the periodic orbit
intersect transversally near the planar homoclinic.
It follows that there are chaotic dynamics in the energy level of the 
out of plane orbit.  
\end{itemize}  

We label $1$ and $2$ with a P, as these are theorems whose proofs 
are already in the literature.  Indeed these are simply 
restatements of the theorems of Henard \cite{MR0365628} and Devaney 
\cite{MR0442990} respectively.  Point $3$ is labeled with a $C$ as, to the 
best of our knowledge this point is conjecture.  
Indeed, the conjecture may be false without further clarification by 
other hypotheses, however 
the computations discussed 
in the present work illustrate that there appear to be situations where it holds.

A proof of $3$ is far beyond the scope of the present -- largely numerical -- work.
Though we provide the following remarks 
outlining an argument which we believe could be made precise with 
appropriate refinments.  First we note that the equilibrium
satisfies the hypothesis of the Lyapunov center theorem -- see for example
\cite{MR0021186,MR0096021,MR1345153} -- since there is only one center
direction and the other directions at the equilibrium are hyperbolic
(assuming that the vector field and the first integral are analytic).  
This guarantees the existence of the out of plane family of periodic orbits.  
The saddle-focus stability of the out of plane orbits follows from the 
center-stable manifold theorem \cite{MR0221044,MR635782}, and from the 
saddle-focus stability follows the claim about the dimension of the stable/unstable
manifolds of the periodic orbit.

The existence of a transverse connecting orbit could be completed by 
formulating the connecting orbit as the solution of a two point boundary 
value problem (BVP), with boundary conditions projected onto the stable/unstable 
manifolds of the periodic orbit.  See \cite{MR2511084} for more complete 
discussion of the BVP formulation of the connecting orbit.  An approximate 
connecting orbit is obtained by taking a suitable portion of the planar
homoclinic.  If the approximation is ``good enough'' then there is hope that
an application of the Newton-Kantorovich theorem \cite{MR0231218} 
could complete the proof.

We stress however that even if the above outline were completed
it would provide results only in a, possibly very, small neighborhood of the invariant plane.  
The numerical results given in the present work on the other hand 
suggest that the planar homoclinics can have an important organizing
effect on the dynamics even for Lyapunov orbits with large out of 
plane amplitude.  At least this appears to be the case for the spatial 
equilateral restricted four body problem.  

A more interesting topic of future work would be to refine the numerical results of 
the present work into theorems for explicit larger out of plane amplitudes. 
 Most likely this would be done using computer
assisted methods of proof.  For example a method for proving the existence of spatial periodic 
orbits for the CRFBP has already been given in \cite{jpJaimeAndMe}, where indeed
the existence of many out of plane orbits coming from the vertical Lyapunov family have 
already been established using computer assisted means.    
Using the methods of \cite{poManProofs} -- or some modification of these -- 
one could compute validated bounds on the attached stable/unstable manifolds
of these periodic orbits.  Once the stable/unstable manifold validations are validated then
computer assisted proof of the desired transverse homoclinic connections can be given 
using small modifications of the techniques developed in \cite{MR3207723,MR3353132}.
Implementing the computer assisted argument just sketched is the topic of 
a work in preparation by the authors.

\begin{acknowledgements}
The second author was partially supported by NSF grant 
DMS-1813501.
Both authors were partially supported by NSF grant DMS-1700154 
and by  the Alfred P. Sloan Foundation grant G-2016-7320.
\end{acknowledgements}

\appendix


\section{Obtaining a polynomial field by automatic differentiation of the CRFBP} \label{sec:autoDiff}

To facilitate formal series calculations in the CRFBP, we first rewrite the problem as a first order ordinary differential equation
 and then introduce a change of variable, often referred to as automatic differentiation, to obtain a polynomial vector field. 
 The problem is recovered via projection, as long as the initial conditions are restricted to an appropriate submanifold. 
We first set
\begin{equation}\label{eq:AutomaticDiffeqn1}
 u_1 = x, \quad u_2 = \dot x, \quad u_3= y, \quad u_4 = \dot y, \quad u_5=z,\quad u_6= \dot z,
\end{equation}
and obtain a first order ODE $\dot u = f(u)$ given by
\begin{equation} \label{eq:FirstOrderOriginal}
\begin{split}
\dot{u_1} &= u_2, \\
\dot{u_2} &= 2 u_4 +\Omega_{u_1},\\
\dot{u_3} &= u_4, \\
\dot{u_4} &= -2u_2 +\Omega_{u_3},\\
\dot{u_5} &= u_6, \\
\dot{u_6} &= \Omega_{u_5},
\end{split}
\end{equation}
where $\Omega$ is as previously given but using the new set of variable. This vector field still have singularities introduced by the terms corresponding to the inverse of the distance with the primaries, we extend our set of variable using the following definitions

\begin{align}
u_7 = \frac{1}{\sqrt{(x-x_1)^2 +(y-y_1)^2 +(z-z_1)^2}} =  \frac{1}{\sqrt{(u_1-x_1)^2 +(u_3-y_1)^2 +(u_5-z_1)^2}}, \label{eq:AutomaticDiffeqn2} \\
u_8 = \frac{1}{\sqrt{(x-x_2)^2 +(y-y_2)^2 +(z-z_2)^2}} =  \frac{1}{\sqrt{(u_1-x_2)^2 +(u_3-y_2)^2 +(u_5-z_2)^2}}, \label{eq:AutomaticDiffeqn3} \\
u_9 = \frac{1}{\sqrt{(x-x_3)^2 +(y-y_3)^2 +(z-z_3)^2}} =  \frac{1}{\sqrt{(u_1-x_3)^2 +(u_3-y_3)^2 +(u_5-z_3)^2}}. \label{eq:AutomaticDiffeqn4}
\end{align}

Let $U \subset \mathbb{R}^6$ be an open set excluding the primaries. Then, a direct computation provides that for the function $R:U \to \mathbb{R}^9$ given by
\begin{equation}\label{eq:R}
R(u_1,u_2,u_3,u_4,u_5,u_6) = \begin{pmatrix}
u_1 \\ u_2 \\ u_3 \\ u_4 \\ u_5 \\ u_6 \\                
\frac{1}{\sqrt{(u_1-x_1)^2 +(u_3-y_1)^2 +(u_5-z_1)^2}} \\
\frac{1}{\sqrt{(u_1-x_2)^2 +(u_3-y_2)^2 +(u_5-z_2)^2}} \\
\frac{1}{\sqrt{(u_1-x_3)^2 +(u_3-y_3)^2 +(u_5-z_3)^2}}
\end{pmatrix}
\end{equation}
and the polynomial vector field $F \colon \mathbb{R}^9 \to \mathbb{R}^9$ given by 
\begin{align}\label{eq:Fwithoutalpha}
F(u)= 
\begin{pmatrix}
u_2 \\
2u_4 +u_1 +m_1(x_1-u_1)u_7u_7u_7 +m_2(x_2-u_1)u_8u_8u_8 +m_2(x_3-u_1)u_9u_9u_9 \\
u_4 \\
-2u_2 +u_3 +m_1(y_1-u_3)u_7u_7u_7 +m_2(y_2-u_3)u_8u_8u_8 +m_2(x_3-u_1)u_9u_9u_9 \\
u_6 \\
m_1(z_1-u_5)u_7u_7u_7 +m_2(z_2-u_5)u_8u_8u_8 +m_2(z_3-u_5)u_9u_9u_9 \\
(x_1-u_1)u_2u_7u_7u_7 +(y_1-u_3)u_4u_7u_7u_7 +(z_1-u_5)u_6u_7u_7u_7 \\
(x_2-u_1)u_2u_8u_8u_8 +(y_2-u_3)u_4u_8u_8u_8 +(z_2-u_5)u_6u_8u_8u_8 \\
(x_3-u_1)u_2u_9u_9u_9 +(y_3-u_3)u_4u_9u_9u_9 +(z_3-u_5)u_6u_9u_9u_9
\end{pmatrix},
\end{align}
we have the infinitesimal conjugacy 
\begin{equation}\label{eq:autoDiff}
DR(u)f(u) = F(R(u)), \quad \forall u \in U.
\end{equation}
Hence orbits of $u' = F(u)$ have the same dynamics as $x' = f(x)$ after projecting 
onto the first six components.
We note that, as an effect of the change of variable, the new vector field does not have any singularity. Nevertheless,
the dynamics of the two are related only on the graph of $R$, and $R$ caries the singularities of 
$f$.  The following items formalize the remarks just made.

\begin{enumerate}
 \item Let $\pi:\mathbb{R}^9 \to \mathbb{R}^6$ denotes the projection onto the first six coordinates. So that for all $u \in U$ we have $u=\pi(R(u))$ and
\[
\pi(F(R(u))) =f(u). 
\]
Therefore we recover the original problem.
\item The orbits of $f$ are mapped onto orbits of $F$ under $R$ and the graph of $R$ is invariant under the flow of $F$. 
\item If $H:\mathbb{R}^6 \to \mathbb{R}$ is constant along curves solution of the initial system, then $G:\mathbb{R}^9 \to \mathbb{R}$ such that $G(R(u))= H(u)$ for all $u \in U$ is constant along curves solution of the extended problem.
\end{enumerate}

It follows from those remarks that it is possible to find periodic orbits of the four body problem using the vector field $F$. Our goal is to compute stable and unstable manifolds for a periodic orbit $\gamma(t)$ of $\dot{u}=f(u)$, so that we have to show that the associated periodic orbit $\Gamma(t)=R(\gamma(t))$ has the same stability type. This is the object of the next Theorem.  

\begin{theorem}
Let $\gamma(t)$ be a periodic orbit of $\dot{u}=f(u)$ with Floquet multiplier $\lambda$ associated to the tangent bundle $v(t)$. Then $\lambda$ is a Floquet multiplier of the periodic orbit $\Gamma(t)=R(\gamma(t))$, solution to the system $\dot x = F(x)$, moreover $\xi(t)= DR(\gamma(t))v(t)$ is the associated tangent bundle. 
\end{theorem}
\begin{proof}
We first note that $v(t)$ will satisfy
\begin{equation}\label{eq:Bundle}
\dot v (t)= Df(\gamma(t))v(t) -\lambda v(t)
\end{equation}
and that differentiating \eqref{eq:autoDiff} provides
\begin{equation}\label{eq:DiffautoDiff}
D^2R(u)f(u) +DR(u)Df(u) = DF(R(u))DR(u). 
\end{equation}
So that a direct computation provides that
\begin{align*}
\dot\xi(t) &= D^2R(\gamma(t))\dot\gamma(t)v(t) +DR(\gamma(t))\dot v(t) \\
           &= D^2R(\gamma(t))f(\gamma(t))v(t) +DR(\gamma(t))Df(\gamma(t))v(t) -DR(\gamma(t))\lambda v(t)
\end{align*}
where we used the fact that $\gamma(t)$ is a periodic solution of $f$ as well as equation \eqref{eq:Bundle}. Then using \eqref{eq:DiffautoDiff}, $\Gamma(t)=R(\gamma(t))$ and $\xi(t)= DR(\gamma(t))v(t)$, we obtain that
\[
\dot \xi(t)= DF(\Gamma(t))\xi(t) -\lambda \xi(t).
\]
This is the desired result.
\end{proof}

It follows from this result that in the extended system six of the multiplier will be known from the usual theory. The other three are all zeros so that the dimension of the stable and unstable manifolds for any orbits remain unchanged.


\section{Orbit Data} \label{sec:orbitData}
In this section we provide several tables of data meant to make the present work more 
reproducible.  Since our calculations of the connecting orbits utilize fairly sophisticated 
Fourier-Taylor approximations of the local stable/unstable manifolds in the formulation 
of the two point boundary value problems, it is unreasonable to think that the casual reader
would reimplement these calculations. On the other hand, many readers will have
 experience in the use of numerical integrators for problems in celestial mechanics 
and once equipped with the equations of motion it is not unreasonable to think one 
might want to reproduce some of the periodic orbits and connections discussed in the 
present work.  To this end we provide accurate initial conditions which can be integrated 
to reproduce the orbits discussed in the present work.  The resulting orbits could also 
be taken as initial conditions for numerical continuation software packages like 
AUTO or MatCont.  

The table are all organized the same way. In the first column, we give the initial point 
expressed as a six-dimensional vector representing the initial position and momentum. 
The coordinates are given in the following order
\[
P_0 = ( x, \dot x, y, \dot y, z, \dot z).
\]
Then the second column of the table provides $T$ an approximation of the period of the 
periodic orbit starting at the point previously given. The third column is $n$, the number of Floquet 
multiplier with positive real part. Finally, the last column shows $J(P_0)$, the energy level of the 
initial data. We note that that case of interest in this paper is when $n=2$ and the multipliers are 
complex conjugate. To obtain the data, we start by computing the center manifold of each libration 
point to find an initial guess for $P_0$ and $T$. To improve the guess, we numerically integrate the 
approximated periodic orbit and express the result in Fourier coefficients. Then Newton's method is applied to 
obtain a guess for the periodic orbit with defect close to machine precision, for all cases covered by the tables 
it suffices to take $50$ Fourier coefficients. The resulting sequence of Fourier coefficients is then a starting 
point for any continuation method in order to find other members of the family. To construct the table, we used 
a zeroth order predictor-corrector algorithm using Newton's method in the space of Fourier coefficients, in this case 
the frequency is an unknown of the system while the energy level is one of the input of the algorithm. The cases 
of $\mathcal{L}_0$ is given at $m_1=0.4$ and $m_2=0.35$ while the case at $\mathcal{L}_5$ is given with equal masses.

\begin{table}[!t]
\begin{tabular}{| L | L | L | L |}
\hline
P_0 & T & n& J(P_0) \\
\hline
(0.1108,-0.0339,0.1004,-0.0068,0.7468,0.3387) &5.3875 &1 & 2.0 \\
(0.1192,-0.0190,0.1138,-0.0015,0.6975,0.3190) &5.0672 &1 & 2.1 \\
(0.1277,-0.0082,0.1226,0.0016,0.6501,0.2993) &4.7779 & 1 & 2.2 \\
(0.1356,-0.0007,0.1277,0.0034,0.6045,0.2797) &4.5162 & 1 & 2.3 \\
(0.1425,0.0042,0.1303,0.0041,0.5605,0.2601) & 4.2789 & 2 & 2.4 \\
(0.1480,0.0073,0.1309,0.0042,0.5177,0.2405) & 4.0632 & 2 & 2.5 \\
(0.1515,0.0090,0.1303,0.0039,0.4759,0.2209) & 3.8666 & 2 & 2.6 \\
(0.1526,0.0095,0.1289,0.0034,0.4347,0.2013) & 3.6873 & 2 & 2.7 \\
(0.1515,0.0092,0.1270,0.0029,0.3938,0.1817) & 3.5238 & 2 & 2.8 \\
(0.1485,0.0083,0.1248,0.0023,0.3524,0.1620) & 3.3746 & 2 & 2.9 \\
(0.1444,0.0070,0.1224,0.0017,0.3099,0.1421) & 3.2382 & 2 & 3.0 \\
(0.1398,0.0055,0.1199,0.0013,0.2652,0.1220) & 3.1132 & 2 & 3.1 \\
(0.1349,0.0038,0.1174,0.0009,0.2162,0.1014) & 2.9983 & 2 & 3.2 \\
(0.1302,0.0022,0.1150,0.0006,0.1584,0.0797) & 2.8924 & 2& 3.3 \\
(0.1257,0.0005,0.1125,0.0002,0.0716,0.0549) & 2.7944 & 2& 3.4 \\
\hline
\end{tabular}
\caption{Family at $L_0$}
\end{table}

\begin{table}
\begin{tabular}{| L | L | L | L |}
\hline
P_0 & T & n & J(P_0)  \\
\hline
(0.4844,-0.5703,-0.2358,0.0306,0.6981,0.5661) &6.2404 & 1 &1.6\\
(0.5063,-0.5224,-0.2102,0.0344,0.6830,0.5563) &6.1590 & 1 &1.7\\
(0.5312,-0.4749,-0.1858,0.0383,0.6656,0.5437) &6.0753 & 1 &1.8\\
(0.5589,-0.4281,-0.1629,0.0419,0.6458,0.5279) &5.9924 & 1 &1.9\\
(0.5892,-0.3822,-0.1417,0.0448,0.6236,0.5088) &5.9132 & 1 &2\\
(0.6225,-0.3378,-0.1221,0.0452,0.5974,0.4883) &5.8397 & 1 &2.1\\
(0.6576,-0.2946,-0.1040,0.0442,0.5679,0.4649) &5.7732 & 1 &2.2\\
(0.6940,-0.2525,-0.0873,0.0416,0.5344,0.4385) &5.7140 & 2 &2.3\\
(0.7313,-0.2113,-0.0717,0.0376,0.4964,0.4085) &5.6618 & 2 &2.4\\
(0.7689,-0.1712,-0.0571,0.0323,0.4528,0.3744) &5.6162 & 2 &2.5\\
(0.8068,-0.1319,-0.0433,0.0259,0.4019,0.3351) &5.5764 & 2 &2.6\\
(0.8447,-0.0934,-0.0303,0.0184,0.3405,0.2884) &5.5417 & 2 &2.7\\
(0.8824,-0.0556,-0.0178,0.0101,0.2616,0.2297) &5.5115 & 2 &2.8\\
(0.9195,-0.0180,-0.0057,0.0016,0.1408,0.1423) &5.4852 & 2 &2.9\\
\hline
\end{tabular}
\caption{Family at $L_5$, this table is computed with equal masses and 
we recall that periodic orbits at $\mathcal{L}_{4,6}$ can be obtained by a 
rotation of $\pm 120$ degrees. The cases of energy from $2.4$ to $2.50$ 
have real Floquet multipliers while the remaining of the table are complex conjugate.}
\end{table}

The connecting orbits in Figure \ref{fig:L0spacial} are homoclinic and accumulate to the periodic orbit with initial condition in the row $J=3.2$, given by the table for $\mathcal{L}_0$. The initial data with higher accuracy is
\begin{align*}
P_0= \begin{pmatrix}  0.134934339930888  \\ 0.003888013139251 \\  0.117443350170703 \\  0.000936082833871 \\ 0.216240831347475  \\ 0.101389225000425 \end{pmatrix}, \quad
T= 2.998307362412966.
\end{align*}
To reproduce the trajectories displayed one can integrate the following initial values $P_0$ back and forward in time for the given time $T$. the starting and ending point of the resulting trajectories will lay on the boundary of the parameterized unstable and stable manifold respectively.
\begin{align*}
P_0= \begin{pmatrix} -0.585194841158983 \\ 0.650674788263036 \\ -0.242897059971999 \\ -0.809665850514842 \\ 0.015366927308435 \\ 0.645609803647810 \end{pmatrix}, \quad
T=  3.9083,\\
P_0= \begin{pmatrix} -0.010028232796882 \\  0.018905042025788 \\ -0.527614375771166 \\ -0.278204402447460 \\ 0.066684862193223 \\  -0.421303149345866 \end{pmatrix}, \quad
T= 3.5848,\\
P_0 = \begin{pmatrix} 0.364232983907004  \\ 0.282004601213298  \\ 0.365277731376544  \\ 0.566929250936993 \\ 0.188719573086921 \\ -0.192150484392393 \end{pmatrix}, \quad
T= 4.1378. \\
\end{align*}

The three connecting orbit accumulating to the same periodic orbit and member of the families displayed in Figure \ref{fig:homoclinicL0continued} can be found using the following initial condition and integration time.
\begin{align*}
P_0= \begin{pmatrix}  -0.101146445518484 \\  0.039260485918423 \\  0.357723970145646  \\ 0.064390124937530 \\ -0.215188925518734 \\ -0.117478748772784 \end{pmatrix}, \quad
T= 2.3112, \\
P_0= \begin{pmatrix} 0.292042336892103  \\ 0.003508935985276 \\  0.118262267817677 \\ -0.031322268117327 \\ 0.009128811923180 \\ -0.481727032516309 \end{pmatrix}, \quad
T= 1.7643, \\
P_0= \begin{pmatrix} -0.082031603660355 \\ -0.244810917818636 \\ -0.371129071110934 \\ -0.141117360021255 \\ 0.090892927654736 \\ -0.410662336419204 \end{pmatrix}, \quad
T= 2.6543.
\end{align*}

In the case of $L_5$, the connection computed are at $J= 2.9$, the initial data for the periodic orbit are given with higher accuracy by
\begin{align*}
P_0= \begin{pmatrix} 0.919523300342616 \\ -0.018021865086785 \\ -0.005720721776858  \\ 0.001586045655911 \\ 0.140748196680255 \\ 0.142288965593486 \end{pmatrix}, \quad
T= 5.485186773053060.
\end{align*}
The midpoint of each connecting orbit as well as the approximate integrating time needed to reach the boundary of the parameterized manifolds are given by pairs, corresponding to their shape and the figure in which they were presented. The initial data for the connecting orbit displayed in Figure \ref{fig:homoclinicL5B} are given by
\begin{align*}
P_0= \begin{pmatrix} -0.221338679671589 \\ 0.290535064047762 \\ 0.807520893403199 \\ -0.079212158161279 \\ 0.099168243248453  \\ -0.192275633578254 \end{pmatrix}, \quad
T= 3.9267,\\
P_0= \begin{pmatrix} -0.027278885368683 \\ -0.415243675715196 \\ -0.681730123750280 \\ -0.689462085636593 \\ 0.002992463395721 \\ 0.060975647933203 \end{pmatrix}, \quad
T= 4.1225.
\end{align*}

The initial data for the connecting orbit displayed in Figure \ref{fig:homoclinicL5A} are given by
\begin{align*}
P_0= \begin{pmatrix}  -0.093216716467939  \\ 0.539629163160594  \\ 0.029112774416657  \\ 0.487738555586829 \\ 0.028678261003714 \\ -0.264834109382665 \end{pmatrix}, \quad
T= 5.7103, \\
P_0= \begin{pmatrix} -0.369576889105909 \\ -0.085470029306448 \\  0.463741869935280 \\  0.545543943960734 \\ 0.088979849325104  \\ -0.103325348379878 \end{pmatrix}, \quad
T= 4.7319.
\end{align*}

The initial data for the connecting orbit displayed in Figure \ref{fig:homoclinicL5} are given by
\begin{align*}
P_0= \begin{pmatrix} 0.497723454800157 \\  0.803131159532739  \\ 1.346319122336476 \\ -0.067118235690442 \\ 0.029291317637547 \\ -0.145379858982222 \end{pmatrix}, \quad
T= 4.9363, \\
P_0= \begin{pmatrix} 0.480703865397053 \\  -0.766478203112893 \\ -1.325717528617470 \\ -0.049652453045623 \\ 0.255108298576897 \\ -0.003998765894041 \end{pmatrix}, \quad
T= 4.9277. \\
\end{align*}

\bibliographystyle{plain}
\bibliography{papers} 

\end{document}